\documentclass[a4paper, 12pt]{article}
\pdfoutput=1

\usepackage[T1]{fontenc}
\usepackage[utf8]{inputenc}

\usepackage{amsfonts,amsmath,amssymb,mathtools,dsfont,microtype} % Symbols & microtypography
\usepackage[usenames, dvipsnames]{xcolor} % Colours with names -- see https://www.overleaf.com/learn/latex/Using_colours_in_LaTeX for list
\usepackage[noBBpl]{mathpazo} % Palatino, but not for mathbb

\DeclareFontFamily{U} {cmr}{}

\DeclareFontShape{U}{cmr}{m}{n}{
	<-6> cmr5
	<6-7> cmr6
	<7-8> cmr7
	<8-9> cmr8
	<9-10> cmr9
	<10-12> cmr10
	<12-> cmr12}{}

\DeclareSymbolFont{Xcmr} {U} {cmr}{m}{n}
\DeclareMathSymbol{\Delta}{\mathord}{Xcmr}{'001}
\DeclareMathSymbol{\Upsilon}{\mathord}{Xcmr}{'007}
\DeclareMathSymbol{\Omega}{\mathord}{Xcmr}{'012}

\usepackage[labelsep = period, labelfont = bf, justification = centering]{caption}
\usepackage{float,graphicx,subcaption}
\usepackage{booktabs,makecell}
\usepackage{enumitem,mdwlist}
\setlist[itemize]{topsep=0ex,itemsep=0ex,parsep=0.4ex}
\setlist[enumerate]{topsep=0ex,itemsep=0ex,parsep=0.4ex}

\usepackage[left = 2.5cm, right = 2.5cm, top = 2.5cm, bottom = 2.5cm, headsep = 12pt, headheight = 15pt]{geometry}
\usepackage{parskip}

\usepackage[hyphens]{url} % Urls together with line breaking them
\usepackage[linktoc = all, hidelinks, colorlinks, unicode=true, pagebackref=true]{hyperref} % Must be loaded after url
\usepackage[capitalise, compress, nameinlink, noabbrev]{cleveref} % Must be loaded after hyperref
\hypersetup{linkcolor={blue!70!black}, citecolor={green!70!black}, urlcolor={blue!70!black}}

\usepackage{amsthm,thmtools,thm-restate}
\declaretheorem[name = Theorem, numberwithin = section, style = plain]{theorem}
\declaretheorem[name = Claim, numberwithin = theorem, style = plain]{claim}

\declaretheorem[name = Conjecture, numberlike = theorem, style = plain]{conjecture}
\declaretheorem[name = Definition, numberlike = theorem, style = definition]{definition}
\declaretheorem[name = Lemma, numberlike = theorem, style = plain]{lemma}
\declaretheorem[name = Observation, numberlike = theorem, style = plain]{observation}
\declaretheorem[name = Proposition, numberlike = theorem, style = plain]{proposition}

\declaretheorem[name = Remark, numberlike = theorem, style = definition]{remark}
\crefname{observation}{Observation}{Observations}
\crefname{conjecture}{Conjecture}{Conjectures}
\crefname{claim}{Claim}{Claims}

\DeclareFontFamily{U}{matha}{\hyphenchar\font45}
\DeclareFontShape{U}{matha}{m}{n}{
	<5> <6> <7> <8> <9> <10> gen * matha
	<10.95> matha10 <12> <14.4> <17.28> <20.74> <24.88> matha12
}{}
\DeclareSymbolFont{matha}{U}{matha}{m}{n}

\DeclareMathSymbol{\specialuparrow}{\mathrel}{matha}{"D2}
\DeclareMathSymbol{\specialrightarrow}{\mathrel}{matha}{"D1}

\renewcommand*{\backref}[1]{}
\renewcommand*{\backrefalt}[4]{
	\ifcase #1 Not cited.%
	\or $\specialuparrow$#2%
	\else $\specialuparrow$#2%
	\fi%
}

\renewcommand{\epsilon}{\varepsilon}
\renewcommand{\ge}{\geqslant}
\renewcommand{\le}{\leqslant}
\renewcommand{\geq}{\geqslant}
\renewcommand{\leq}{\leqslant}
\renewcommand{\emptyset}{\varnothing}

\DeclarePairedDelimiter{\abs}{\lvert}{\rvert}

\DeclarePairedDelimiter{\floor}{\lfloor}{\rfloor}
\DeclarePairedDelimiter{\set}{\{}{\}}

\newcommand{\defn}[1]{\textcolor{Maroon}{\emph{#1}}}

\newcommand{\Ztwo}{\mathbb{Z}/2\mathbb{Z}}
\newcommand{\Zk}{\mathbb{Z}/k\mathbb{Z}}
\newcommand{\Zr}{\mathbb{Z}/\rho\mathbb{Z}}
\newcommand{\ZNN}{\mathbb{Z}_{\geqslant 0}}
\newcommand{\db}{d^{\ast}}
\newcommand{\di}{d^{I_\infty}}
\newcommand{\da}{\bar{d}}
\newcommand{\dsup}{d_{\sup}^{I_\infty}}
\newcommand{\notempty}{\neq \varnothing}

\newcommand{\W}{\mathbf{W}}

\newcommand{\bN}{\mathbb{N}}
\newcommand{\bP}{\mathbb{P}}

\newcommand{\bZ}{\mathbb{Z}}
\newcommand{\cA}{\mathcal{A}}
\newcommand{\cF}{\mathcal{F}}
\newcommand{\cO}{\mathcal{O}}

\newcommand{\fF}{F}
\newcommand{\fG}{G}
\newcommand{\dsupG}{d_{\sup\fG}^{I_\infty}}

\title{The structure and density of $k$-product-free sets in the free semigroup}
\date{7 July 2023}

\begin{document}

\author{
Freddie Illingworth\footnotemark[2]\qquad Lukas Michel\footnotemark[2]\qquad Alex Scott\footnotemark[2]
}

\maketitle

\begin{abstract}
	The free semigroup $\cF$ over a finite alphabet $\cA$ is the set of all finite words with letters from $\cA$ equipped with the operation of concatenation. A subset $S$ of $\cF$ is $k$-product-free if no element of $S$ can be obtained by concatenating $k$ words from $S$, and strongly $k$-product-free if no element of $S$ is a (non-trivial) concatenation of at most $k$ words from $S$.

    We prove that a $k$-product-free subset of $\cF$ has upper Banach density at most $1/\rho(k)$, where $\rho(k) = \min\set{\ell \colon \ell \nmid k - 1}$. We also determine the structure of the extremal $k$-product-free subsets for all $k \notin \set{3, 5, 7, 13}$; a special case of this proves a conjecture of Leader, Letzter, Narayanan, and Walters. We further determine the structure of all strongly $k$-product-free sets with maximum density. Finally, we prove that $k$-product-free subsets of the free group have upper Banach density at most $1/\rho(k)$, which confirms a conjecture of Ortega, Ru\'{e}, and Serra.
\end{abstract}

\renewcommand{\thefootnote}{\fnsymbol{footnote}} % Make affiliation marks symbols


\footnotetext[2]{Mathematical Institute, University of Oxford, United Kingdom (\textsf{\{\href{mailto:illingworth@maths.ox.ac.uk}{illingworth},\href{mailto:michel@maths.ox.ac.uk}{michel},\href{mailto:scott@maths.ox.ac.uk}{scott}\}@maths.ox.ac.\allowbreak uk}). Research of FI and AS supported by EPSRC grant EP/V007327/1.}

\renewcommand{\thefootnote}{\arabic{footnote}} % Return to normal footnote symbols

\section{Introduction}\label{sec:intro}

A subset $S$ of a (semi)group $G$ is said to be \defn{product-free} if $x \cdot y \notin S$ for all $x, y \in S$. Two very natural questions present themselves.
\begin{enumerate}
    \item[] \textbf{Density:} How dense can the largest product-free subset of $G$ be?
    \item[] \textbf{Structure:} What is the structure of the densest product-free subsets of $G$?
\end{enumerate}
These problems have been extensively studied over the last fifty years. In the finite abelian case, this culminated in a solution to the density problem by Green and Ruzsa~\cite{GR05} and the structure problem by Balasubramian, Prakash, and Ramana~\cite{BPR16}. The finite non-abelian case was first investigated by Babai and S\'{o}s~\cite{BS85}. This case behaves very differently with the possibility of the largest product-free subsets having vanishing density as shown by the seminal work of Gowers~\cite{Gowers08} on quasirandom groups. Recent breakthroughs include the alternating group where Eberhard~\cite{Eberhard16} solved the density problem (up to logarithmic factors) and Keevash, Lifshitz, and Minzer~\cite{KLM22} solved the structure problem. We refer the reader to \cite{Kedlaya09,TV17} for surveys of the area.

In the infinite non-abelian setting, Leader, Letzter, Narayanan, and Walters~\cite{LLNW20} solved the density problem for a \defn{free semigroup}\footnote{The free semigroup on alphabet $\cA$ is the set of all finite words whose letters are in $\cA$ equipped with the associative operation of concatenation and whose identity is the empty word.} $\cF$ on a finite alphabet $\cA$ with respect to the measure that assigns weight $\abs{\cA}^{-n}$ to each word of length $n$. This is the natural measure induced by sampling uniformly random words from $\cF$ and gives total weight 1 to the words of length $n$. As noted in \cite{LLNW20}, the counting measure leads to degenerate results (in particular, intuitively small product-free sets with density close to 1). Leader, Letzter, Narayanan, and Walters solved the density problem proving the following where $\db$ is the \emph{upper Banach density} (see \cref{sec:density} for formal definitions).

\begin{theorem}[\cite{LLNW20}]\label{thm:LLNW}
    Let $\cA$ be a finite set and $\cF$ be the free semigroup with alphabet $\cA$. If $S \subset \cF$ is product-free, then $\db(S) \leq 1/2$.
\end{theorem}

There is a simple class of examples of large product-free subsets of $\cF$ that show that $1/2$ in \cref{thm:LLNW} is best possible. For a non-empty subset $\Gamma \subset \cA$ the \defn{odd-occurrence set} $\cO_\Gamma \subset \cF$ generated by $\Gamma$ is the set of words in which the total number of occurrences of letters from $\Gamma$ is odd (note that if $\Gamma = \cA$, then $\cO_\Gamma$ consists of all words of odd length). It is easy to see that these are product-free with density $1/2$. Leader, Letzter, Narayanan, and Walters conjectured that these are the only examples.

\begin{conjecture}[{\cite{LLNW20}}]\label{conj:LLNW}
    Let $\cA$ be a finite set and $\cF$ be the free semigroup with alphabet $\cA$. If $S \subset \cF$ is product-free and $\db(S) = 1/2$, then $S \subset \cO_\Gamma$ for some nonempty subset $\Gamma \subset \cA$.
\end{conjecture}

We confirm \cref{conj:LLNW} and in fact prove a more general result (\cref{thm:skpfstructure}). Calkin and Erd\H{o}s~\cite{CE96} and \L uczak and Schoen~\cite{LS97} defined a subset $S$ of a (semi)group to be \defn{$k$-product-free} ($k \geq 2$) if $x_1 \cdot \dotsc \cdot x_k \notin S$ for all $x_1, \dotsc, x_k \in S$ and to be \defn{strongly $k$-product-free} if it is $\ell$-product-free for every $\ell = 2, \dotsc, k$. Ortega, Ru\'{e}, and Serra extended \cref{thm:LLNW} to strongly $k$-product-free sets as well as to the free group.

\begin{theorem}[{\cite{ORS23}}]\label{thm:ORS}
    Let $k \geq 2$ be an integer, $\cA$ be a finite set, and $\cF$ be the free \textup{(}semi\textup{)}group with alphabet $\cA$. If $S \subset \cF$ is strongly $k$-product-free, then $\db(S) \leq 1/k$.
\end{theorem}

Our first main theorem solves the structure problem for free semigroups, describing the structure of strongly $k$-product-free sets $S \subset \cF$ with density $1/k$. This confirms \cref{conj:LLNW}. An alternative view of the odd-occurrence set $\cO_\Gamma$ is as follows: label each letter in $\Gamma$ with a $1$ and every other letter with a $0$ and let the sum of a word be the sum of the labels of its letter; $\cO_\Gamma$ is the set of words with odd sum. The natural generalisation of this to $k \geq 3$ provides strongly $k$-product-free subsets of $\cF$ with density $1/k$ (see \cref{rmk:skpf}). We prove that these are the only examples. 

\begin{theorem}\label{thm:skpfstructure}
    Let $k \geq 2$ be an integer, $\cA$ be a finite set, and $\cF$ be the free semigroup with alphabet $\cA$. If $S \subset \cF$ is strongly $k$-product-free and $\db(S) = 1/k$, then the following holds. It is possible to label each letter of $\cA$ with a label in $\Zk$ such that $S$ is a subset of the strongly $k$-product-free set
    \begin{equation*}
        T \coloneqq \set{w \in \cF \colon \text{the sum of the labels of letters in $w$ is } 1 \bmod{k}}.
    \end{equation*}
\end{theorem}

\begin{remark}\label{rmk:skpf}
    If some prime divides $k$ and every label given to letters in $\cA$, then $T$ will be empty. If there is no such prime, then $T$ will be non-empty by Bezout's lemma. If $T$ is non-empty, then $\db(T) = 1/k$. Indeed, let $\alpha_1 \alpha_2 \dotsb$ be an infinite random word where the $\alpha_i$ are independent uniformly random letters from $\cA$ and let $X_n$ be the sum of the labels of $\alpha_1$, $\alpha_2$, \ldots, $\alpha_n$. Then $(X_n)$ is a Markov chain on $\Zk$ that is irreducible (since $T \notempty$). The uniform distribution $\pi$ on $\Zk$ is stationary for this chain. Let $d$ be the period of $(X_n)$: by the Markov convergence theorem, for each fixed $r \in \set{1, \dotsc, k - 1}$, the subsequence $(X_{nd + r})$ converges to $\pi$ in distribution, and so the averages $\abs{I}^{-1} \sum_{n \in I} X_n$ over long intervals converge to $\pi$ in distribution. In particular, $\db(T) = \pi(1) = 1/k$.
\end{remark}

We now turn to $k$-product-free sets. In the special case $\abs{\cA} = 1$, the free semigroup $\cF$ is isomorphic to the non-negative integers under addition. In this case, the term `sum-free' is used in place of `product-free'. Calkin and Erd\H{o}s~\cite{CE96} conjectured that a $k$-sum-free subset of the non-negative integers has density at most $1/\rho(k)$ where \defn{$\rho(k)$} is
\begin{equation*}
    \rho(k) \coloneqq \min \set{\ell \in \bZ^+ \colon \ell \nmid k - 1}.
\end{equation*}
Note that the integers which are $1 \bmod{\rho(k)}$ form a $k$-product-free set and so $1/\rho(k)$ would be best possible. \L uczak and Schoen~\cite{LS97} confirmed this conjecture and also solved the structure problem for non-negative integers. We extend their results by solving both the density problem (for all $k$) and the structure problem (provided $k \notin \set{3, 5, 7, 13}$) for $k$-product-free subsets of the free semigroup.

\begin{theorem}\label{thm:kpf}
    Let $k \geq 2$ be an integer, $\cA$ be a finite set, and $\cF$ be the free semigroup with alphabet $\cA$. If $S \subset \cF$ is $k$-product-free, then $\db(S) \leq 1/\rho(k)$.
\end{theorem}

\Cref{thm:kpfstructure} shows that the structure of the extremal $k$-product-free sets is very similar to that of strongly $k$-product-free sets except everything is modulo $\rho(k)$. See \cref{sec:open} for further discussion of the cases when $k$ is 3, 5, 7, or 13.

\begin{theorem}\label{thm:kpfstructure}
    Let $k \geq 2$ be an integer with $k \notin \set{3, 5, 7, 13}$ and $\rho = \rho(k)$. Let $\cA$ be a finite set and $\cF$ be the free semigroup with alphabet $\cA$. If $S \subset \cF$ is $k$-product-free and $\db(S) = 1/\rho$, then the following holds. It is possible to label each letter of $\cA$ with a label in $\Zr$ such that $S$ is a subset of the $k$-product-free set
    \begin{equation*}
        T \coloneqq \set{w \in \cF \colon \text{the sum of the labels in $w$ is } 1 \bmod{\rho}}.
    \end{equation*}
\end{theorem}

Note, just as in \cref{rmk:skpf}, that if some prime divides $\rho$ and every label given to a letter in $\cA$, then $T$ is empty. Otherwise $T$ is non-empty, $k$-product-free, and has density $1/\rho(k)$.

Finally, we consider the free group. \Cref{thm:ORS} solves the density problem for strongly $k$-product-free sets. Ortega, Ru\'{e}, and Serra~\cite{ORS23} made a conjecture corresponding to Calkin and Erd\H{o}s's for $k$-product-free sets. We prove this conjecture.

\begin{theorem}\label{thm:kpffree}
    Let $k \geq 2$ be an integer, $\cA$ be a finite set, and $\fF$ be the free group with alphabet $\cA$. If $S \subset \fF$ is $k$-product-free, then $\db(S) \leq 1/\rho(k)$.
\end{theorem}

The rest of the paper is structured as follows. In \cref{sec:density} we provide the formal definitions of density. In \cref{sec:technical} we prove some important technical lemmas and state our main density result, \cref{thm:pfdensity}, from which \cref{thm:kpf} follows. Before proving \cref{thm:pfdensity} we obtain our structural results whose proofs are simpler and already contain some of the key ideas. The proof of \cref{thm:skpfstructure} is given in \cref{sec:skpfstructure} and the proof of \cref{thm:kpfstructure} in \cref{sec:kpfstructure}. In \cref{sec:steeple} we build the machinery that we use to prove \cref{thm:pfdensity} in \cref{sec:pfdensity}. In \cref{sec:pfdensityfree} we adapt our arguments to the free group. We finish, in \cref{sec:open}, with some open problems.

\section{Density}\label{sec:density}

Throughout this paper $\cF$ will be the free semigroup on a finite alphabet $\cA$. To motivate and provide intuition for the notation we view things from the perspective of a randomly generated word. Let $\W = \alpha_1 \alpha_2 \dotsb$ be a random infinite word where each $\alpha_i$ is an independent uniformly random letter in $\cA$. Taking $\W_n = \alpha_1 \alpha_2 \dotsb \alpha_n$, we may view $(\W_n)$ as a random walk on the infinite $\abs{\cA}$-ary tree. We say $\W$ \defn{hits} a set $B \subset \cF$ if the random walks hits $B$ (equivalently if $\W$ has a prefix in $B$) and $\W$ \defn{avoids} $B$ otherwise. We equip $\cF$ with a measure \defn{$\mu$} satisfying, for every word $w \in \cF$,
\begin{equation*}
    \mu(w) = \bP(\W \text{ hits } w) = \abs{\cA}^{-\abs{w}}.
\end{equation*}
Note that, for $B \subset \cF$, $\mu(B) = \sum_{w \in B} \mu(w)$ is the expected number of times that $\W$ hits $B$. This has a useful corollary. A set $C \subset \cF$ is \defn{prefix-free} if there are not distinct words $a, b \in C$ where $a$ is a prefix of $b$. $\W$ can hit a prefix-free set at most once.

\begin{observation}\label{obs:measureC}
    If $C \subset \cF$ is prefix-free, then $\mu(C) \leq 1$.
\end{observation}

For a positive integer $n$ and a set $B \subset \cF$ the \defn{length $n$ layer of $B$} is
\begin{equation*}
    B(n) \coloneqq \set{w \in B \colon \abs{w} = n},
\end{equation*}
while, for an interval $I \subset \bZ^+$,
\begin{equation*}
    B(I) \coloneqq \set{w \in B \colon \abs{w} \in I}.
\end{equation*}
Note that the measure $\mu$ is defined so that $\mu(\cF(n)) = 1$. The density of $B$ on layer $n$ is $\abs{B(n)}/\abs{\cF(n)} = \mu(B(n))$, which is the probability that $\W_n$ is in $B$. The \defn{density of $B$ on interval $I$} is
\begin{equation*}
    d^I(B) \coloneqq \frac{\mu(B(I))}{\mu(\cF(I))} = \abs{I}^{-1} \sum_{n \in I} \mu(B(n)).
\end{equation*}
With these definitions in place, we may give standard notions of density. The \defn{upper asymptotic density} of $B$ is
\begin{equation*}
    \da(B) \coloneqq \limsup_{m \to \infty} d^{\set{1, 2, \dotsc, m}}(B) = \limsup_{m \to \infty} \sum_{n = 1}^m \mu(B(n))/m.
\end{equation*}
The \defn{upper Banach density} of $B$ is
\begin{equation*}
    \db(B) \coloneqq \limsup_{I \to \infty} d^I(B) = \limsup_{I \to \infty} \ \abs{I}^{-1}\sum_{n \in I} \mu(B(n)),
\end{equation*}
where $I$ is an interval and the notation $I \to \infty$ denotes that both $\abs{I}$ and $\min I$ tend to infinity\footnote{The condition $\min I \to \infty$ is often omitted from the definition. However, some simple analysis shows that, whether or not this condition is included, the resulting density is the same.}. Now $\db(B) \geq \da(B)$ for any set $B$ and so all of our results also hold for asymptotic density.

It should be noted that limit superiors are only subadditive (and not additive). In particular, for disjoint sets $A, B \subset \cF$ we have $\db(A \cup B) \leq \db(A) + \db(B)$ and equality may not hold. As an example, the sets
\begin{align*}
    A & = \bigcup_{n \in \bZ^+} \cF(\set{(2n - 1)! + 1, (2n - 1)! + 2, \dotsc, (2n)!}), \\
    B & = \bigcup_{n \in \bZ^+} \cF(\set{(2n)! + 1, (2n)! + 2, \dotsc, (2n + 1)!})
\end{align*}
are disjoint and both have density 1.

Despite this, in the group of non-negative integers $\cF = \bZ^+$, $\db(B)$ satisfies some useful properties. For example, it holds that $\abs{d^I(x + B) - d^I(B)} \le x / \abs{I}$. This implies that $\db(x + B) = \db(B)$. Even more importantly, if $x_1, \dots, x_n \in \bZ^+$ are such that $x_1 + B, \dots, x_n + B$ are disjoint, then $d^I(x_1 + B) + \dots + d^I(x_n + B) \le 1$, implying that
\begin{equation*}
    n \cdot d^I(B) \le \sum_{i = 1}^n \left(d^I(x_i + B) + \frac{x_i}{\abs{I}}\right) \le 1 + \sum_{i = 1}^n \frac{x_i}{\abs{I}}
\end{equation*}
and so $n \cdot \db(B) \le 1$. Not only can this provide upper bounds on the density of $B$, but if we knew that $\db(B) > 1/n$, we could conclude that the sets $x_1 + B, \dots, x_n + B$ cannot all be disjoint and thereby deduce some structural information about $B$. Such arguments were used by \L uczak and Schoen~\cite{Luczak95,LS97} for their results about sum-free subsets of the non-negative integers.

If $\abs{\cA} > 1$, these arguments no longer work. For example, if $w \in \cF$, it is easy to see that $\db(w B) = \abs{\cA}^{-\abs{w}} \cdot \db(B)$ where $w B \coloneqq \set{w b \colon b \in B}$. Also, the fact that $w_1 B, \dots, w_n B$ are disjoint gives no general upper bound on the density of $B$. Even if we consider nested sets $B, w B, \dots, w^n B$, taking $B \coloneqq \cF \setminus (w \cF)$ provides an example where these sets are pairwise disjoint, but $\db(B) = 1 - \abs{\cA}^{-\abs{w}}$ which can be arbitrarily close to $1$.

We address these issues in the next section. By modifying the density that we consider, we can ensure that the density is additive. Importantly, the density of the set $S \subset \cF$ whose upper Banach density we want to bound will not change. Moreover, in certain situations, we prove that $n$ disjoint nested copies of $B$ imply that the density of $B$ is at most $1/n$. This will be crucial for proving our structural results.

\section{Diagonalisation and relative density}\label{sec:technical}

Throughout the paper $S \subset \cF$ will be a fixed set whose upper Banach density we wish to bound (for example, $S$ might be $k$-product-free). There is a sequence of intervals $(I_j)$ such that $I_j \to \infty$ and
\begin{equation*}
    d^{I_j}(S) \to \db(S), \quad \text{as } j \to \infty.
\end{equation*}
Let $B \subset \cF$ be another set. The sequence $(d^{I_j}(B))$ is bounded (all terms are in $[0, 1]$) and so, by the Bolzano-Weierstrass theorem, has a convergent subsequence. In particular, by passing to a subsequence of $(I_j)$, we may assume that $d^{I_j}(S) \to \db(S)$ and $(d^{I_j}(B))$ converges to some limit that we will call \defn{$\di(B)$}. Given a countable collection of subsets of $\cF$, we may, by a diagonalisation argument, assume there is a subsequence $(I_j)$ such that $d^{I_j}(B) \to \di(B)$ for every $B$ in the collection where $\di(S) = \db(S)$. Throughout this paper we will only ever consider countably many sequences and so this convergence occurs for all sets we consider. These limits, unlike the corresponding upper Banach densities, are additive. Indeed, if sets $A$ and $B$ are disjoint, then $d^{I_j}(A \cup B) = d^{I_j}(A) + d^{I_j}(B)$ and so $\di(A \cup B) = \di(A) + \di(B)$. It should be noted that while $\di(S) = \db(S)$, we only have $\di(B) \leq \db(B)$ for the other sets that we consider.

For our structural proofs we will need not only to bound the density of a product-free set $S$ but also to bound the density of $S$ on subtrees. We now begin to define this.

The product \defn{AB} of two sets $A, B \subset \cF$ is
\begin{equation*}
    AB \coloneqq \set{ab \colon a \in A, b \in B}
\end{equation*}
and the set $B^k$ is the product of $k$ copies of $B$. Note that $B$ is $k$-product-free exactly if $B \cap B^k = \emptyset$. A particular important example of a product is $w \cF$ for a word $w \in \cF$: this is exactly the subtree of $\cF$ consisting of all words starting with $w$. Similarly $B \cF$ is exactly the set of words that have a prefix in $B$.

For a finite set $B \subset \cF$ we write \defn{$\min B$} and \defn{$\max B$} for the length of the shortest and longest words in $B$, respectively. Note that if $B$ is finite, then for $n \geq \max B$ the random infinite word $\W$ hits $(B\cF)(n)$ if and only if it hits $B$.

\begin{observation}\label{obs:densityCF}
    If $n \geq \abs{w}$, then $\mu((w \cF)(n)) = \mu(w)$. If $C \subset \cF$ is prefix-free and finite, then $\mu((C \cF)(n)) = \mu(C)$ for $n \geq \max C$.
\end{observation}

\begin{definition}[relative density]
    Let $w \in \cF$ and $B \subset \cF$. For $n \geq \abs{w}$, the relative density of $B$ in $w \cF$ on layer $n$ is
    \begin{equation*}
        \frac{\abs{B(n) \cap w \cF}}{\abs{\cF(n) \cap w \cF}} = \frac{\mu(B(n) \cap w \cF)}{\mu(\cF(n) \cap w \cF)} = \frac{\mu(B(n) \cap w \cF)}{\mu(w)}
    \end{equation*}
    which is the probability that $\W_n$ is in $B$ conditioned on the event that $\W$ hits $w$. If $n < \abs{w}$, then we will take the relative density to be $0$ by convention.

    Furthermore, if $I$ is an interval with $\min I \geq \abs{w}$, then the \defn{relative density of $B$ in $w \cF$ on interval $I$} is
    \begin{equation*}
        d_{w\cF}^I(B) \coloneqq \frac{\mu(B(I) \cap w \cF)}{\mu(\cF(I) \cap w \cF)} = \abs{I}^{-1} \mu(w)^{-1} \sum_{n \in I} \mu(B(n) \cap w\cF) = \mu(w)^{-1} \cdot d^I(B \cap w\cF).
    \end{equation*}
    If $\min I < \abs{w}$, then we will take the relative density to be $0$ by convention.
\end{definition}

Note that if $w$ is the empty word then this relative density is just $d^I(B)$.

Consider the sequence of intervals $(I_j)$ given above where $d^{I_j}(B) \to \di(B)$ for every set $B$ in a countable collection. For each word $w \in \cF$ and each set in the collection, the sequence $(d_{w \cF}^{I_j}(B))$ is bounded (all terms are in $[0, 1]$) and so, by the Bolzano-Weierstrass theorem, has a convergent subsequence. Since $\cF$ is countable (it consists of only finite words) we may, via a diagonalisation argument, pass to a subsequence $(I_j)$ such that, for every $w \in \cF$ and every $B$ in the countable collection, $(d_{w \cF}^{I_j}(B))$ converges to some limit \defn{$\di_{w\cF}(B)$}. In conclusion, we may assume throughout the paper that for any set $B$ we encounter and for all $w \in \cF$ we have
\begin{equation*}
    d_{w \cF}^{I_j}(B) \to \di_{w\cF}(B),
\end{equation*}
where $\di(B) \leq \db(B)$ and $\di(S) = \db(S)$ for one fixed set $S$. As before, these limits are additive. They satisfy the useful property that we may strip away prefixes.

\begin{lemma}\label{lem:densitycancel}
    If $w, v \in \cF$, then $\di_{w v \cF}(w B) = \di_{v \cF}(B)$.
\end{lemma}

\begin{proof}
    Let $I$ be any interval with $\min I > \abs{wv}$. Now
    \begin{equation*}
        d^I_{wv\cF}(wB) = \abs{I}^{-1} \mu(wv)^{-1} \sum_{n \in I} \mu((wB)(n) \cap wv\cF).
    \end{equation*}
    Removing the leading $w$ from each word in $(wB)(n) \cap wv\cF$ shows that $\mu((wB)(n) \cap wv\cF) = \mu(w) \cdot \mu(B(n - \abs{w}) \cap v\cF)$. Also $\mu(wv) = \mu(w)\mu(v)$ and so
    \begin{equation*}
        d^I_{wv\cF}(wB) = \abs{I}^{-1} \mu(v)^{-1} \sum_{n \in I - \abs{w}} \mu(B(n) \cap v \cF),
    \end{equation*}
    where $I - \abs{w}$ is the interval obtained by subtracting $\abs{w}$ from each element of $I$. Thus
    \begin{equation*}
        \abs{d^I_{wv\cF}(wB) - d^I_{v \cF}(B)} = \abs{I}^{-1} \mu(v)^{-1} \cdot \abs[\bigg]{\sum_{n \in I - \abs{w}} \mu(B(n) \cap v \cF) - \sum_{n \in I} \mu(B(n) \cap v \cF)}
    \end{equation*}
    But, for each integer $n$, $\mu(B(n) \cap v \cF) \in  [0, 1]$ and so
    \begin{equation*}
        \abs{d^I_{wv\cF}(wB) - d^I_{v \cF}(B)} \leq \abs{I}^{-1} \mu(v)^{-1} \cdot \abs{w}
    \end{equation*}
    Setting $I = I_j$ and taking $j$ to infinity gives the required result.
\end{proof}

We are now ready to make an important definition that captures the densest that a set $B$ can be down a subtree.

\begin{definition}[sup density]
    For a set $B$ in the countable collection, the \defn{sup density of $B$} is
    \begin{equation*}
        \dsup(B) \coloneqq \sup_{w \in \cF} \di_{w \cF}(B).
    \end{equation*}
\end{definition}

Of course, the sup density satisfies $\dsup(B) \geq \di(B)$ (note that the empty word is in $\cF$) and so $\dsup(S) \geq \db(S)$.

We will prove the following strengthening of \cref{thm:ORS,thm:kpf} in \cref{sec:pfdensity}.

\begin{theorem}\label{thm:pfdensity}
    Let $k \geq 2$ be an integer, $\cA$ be a finite set, and $\cF$ be the free semigroup with alphabet $\cA$.
    \begin{enumerate}[label = {\textup{(\alph{*})}}]
        \item If $S \subset \cF$ is strongly $k$-product-free, then $\db(S) \leq 1/k$. Moreover, if $\db(S) = 1/k$, then $\dsup(S) = 1/k$. \label{thm:skpfdensity}
        \item If $S \subset \cF$ is $k$-product-free, then $\db(S) \leq 1/\rho(k)$. Moreover, if $\db(S) = 1/\rho(k)$, then $\dsup(S) = 1/\rho(k)$. \label{thm:kpfdensity}
    \end{enumerate}
\end{theorem}

This strengthening is needed for our structural results, \cref{thm:skpfstructure,thm:kpfstructure}. For example, if $S \subset \cF$ is strongly $k$-product-free with $\db(S) = 1/k$, then by \ref{thm:skpfdensity}, $\di(S) = \db(S) = \dsup(S)$. This suggests that $S$ is uniformly distributed down subtrees which is made precise by the following lemma.

\begin{lemma}\label{lem:densityuniform}
    If $\di(B) = \dsup(B)$, then $\di_{w\cF}(B) = \di(B)$ for every word $w \in \cF$.
\end{lemma}

\begin{proof}
    Let $\ell$ be a non-negative integer and let $I$ be an interval with $\min I > \ell$. Every word of length greater than $\ell$ is in exactly one $w\cF$ (where $w \in \cF(\ell)$). Hence,
    \begin{equation*}
        d^I(B) = \sum_{w \in \cF(\ell)} d^I(B \cap w \cF) = \sum_{w \in \cF(\ell)} \mu(w) \cdot d^{I}_{w\cF}(B).
    \end{equation*}
    Setting $I = I_j$ and taking $j$ to infinity gives
    \begin{equation*}
        \di(B) = \sum_{w \in \cF(\ell)} \mu(w) \cdot \di_{w \cF}(B).
    \end{equation*}
    Now $\sum_{w \in \cF(\ell)} \mu(w) = \mu(\cF(\ell)) = 1$ and every $w \in \cF(\ell)$ satisfies $\di_{w\cF}(B) \leq \dsup(B) = \di(B)$. Hence we must have $\di_{w\cF}(B) = \di(B)$ for every $w \in \cF(\ell)$. The integer $\ell$ was arbitrary and so we have the required result.
\end{proof}

The next two lemmas are the key technical results for our structural proofs. We remark that for the non-negative integers (that is, when $\abs{\cA} = 1$) they are much more obvious.

\begin{lemma}\label{lem:disjoint}
    Let $S \subset \cF$ be such that $\di(S) = \dsup(S) > 1/n$. Then, for any $w_1, \dotsc, w_n \in \cF$, the sets 
    \begin{equation*}
        w_1 S, \quad w_1 w_2 S, \quad \dotsc, \quad w_1 w_2 \dotsb w_{n-1} S, \quad w_1 w_2 \dotsb w_n S
    \end{equation*}
    cannot be pairwise disjoint.
\end{lemma}

\begin{proof}
    Assume that these sets are pairwise disjoint. Then, for any word $w \in \cF$,
    \begin{equation*}
        \di_{w\cF}(w_1 S) + \dots + \di_{w\cF}(w_1 \dotsb w_n S) = \di_{w\cF}((w_1 S) \cup \dots \cup (w_1 \dotsb w_n S)) \le 1.
    \end{equation*}
    Choose $w = w_1 \dotsb w_n$. Applying \cref{lem:densitycancel} to each term gives
    \begin{equation*}
        \di_{w_2 \dotsb w_n \cF}(S) + \dotsb + \di_{w_n \cF}(S) + \di_\cF(S) \leq 1.
    \end{equation*}
    By \cref{lem:densityuniform}, each term is $\di(S)$ which contradicts $\di(S) > 1/n$, as required.
\end{proof}

\begin{lemma}\label{lem:disjointtwo}
    Let $S \subset \cF$ be such that $\di(S) = \dsup(S) > 2/(2n-1)$. Then, for any $w_1, \dotsc, w_n, v_1, \dotsc, v_n \in \cF$ and $C \subset S$, either the sets 
    \begin{equation*}
        w_1 S, \quad w_1 w_2 S, \quad \dotsc, \quad w_1 \dotsb w_{n-1} S, \quad w_1 \dotsb w_n C
    \end{equation*}
    or the sets
    \begin{equation*}
        v_1 S, \quad v_1 v_2 S, \quad \dotsc, \quad v_1 \dotsb v_{n-1} S, \quad v_1 \dotsb v_n (S \setminus C)
    \end{equation*}
    are not pairwise disjoint.
\end{lemma}

\begin{proof}
    Assume that both collections of sets are pairwise disjoint. Then, as in the proof of \cref{lem:disjoint},
    \begin{equation*}
        \di_{w_2 \dotsb w_n \cF}(S) + \dots + \di_{w_n \cF}(S) + \di_{\cF}(C) \le 1
    \end{equation*}
    and
    \begin{equation*}
        \di_{v_2 \dotsb v_n \cF}(S) + \dots + \di_{v_n \cF}(S) + \di_{\cF}(S \setminus C) \le 1.
    \end{equation*}
    Note that $\di_{\cF}(S \setminus C) = \di_{\cF}(S) - \di_{\cF}(C)$. Applying this and adding the two inequalities, we get
    \begin{align*}
        \di_{w_2 \dotsb w_n \cF}(S) + \dots + \di_{w_n \cF}(S) + \di_{v_2 \dotsb v_n \cF}(S) + \dots + \di_{\cF}(S) \le 2.
    \end{align*}
    However, by \cref{lem:densityuniform}, each term is $\di(S)$ which contradicts $\di(S) > 2/(2n - 1)$, as required.
\end{proof}

\section{Structure of strongly \texorpdfstring{$k$}{k}-product-free sets}\label{sec:skpfstructure}

In this section we prove \cref{thm:skpfstructure} assuming \cref{thm:pfdensity}. Therefore, let $S \subset \cF$ be strongly $k$-product-free satisfying $\db(S) = 1/k$. Note, by \cref{thm:pfdensity}, that $\di(S) = 1/k = \dsup(S)$ and so we may and will frequently apply \cref{lem:disjoint,lem:disjointtwo} with $n = k + 1$.

We want to show that we can label each letter of $\cA$ with a label in $\Zk$ such that $S$ is a subset of
\begin{equation*}
    T \coloneqq \set{a \in \cF \colon \text{the sum of the labels of letters in $a$ is } 1 \bmod{k}}.
\end{equation*}
Assume that each $a \in \cF$ is labelled with this sum. To deduce the structure of $S$, we would like to identify these labels for all words $a \in \cF$. Clearly, everything in $S$ should be labelled $1$. For any other $a \in \cF$, appending a word from $S$ should increase the label by $1$. So, if $a$ has label $\ell$ and we append $i = -\ell \in \Zk$ words from $S$ to $a$, we should get the label $0$, and appending one more word from $S$ should give the label $1$, which might itself be a word from $S$. On the other hand, for any other $j \in \Zk$, appending $j + 1$ words from $S$ to $a$ should give a label different from $1$ and should therefore never yield a word from $S$.

Based on this intuition, for $i = 0, 1, \dotsc, k - 1$ define
\begin{equation*}
    T_i \coloneqq \set{a \in \cF \colon S \cap aS^{i + 1} \notempty}.
\end{equation*}
Then, everything in $T_i$ should have the label $-i \in \Zk$. So, we expect that $S \subset T_{k - 1}$ and that $T_i T_j \subset T_{i + j}$. This is exactly what we will show and which allows us to deduce the structure of $T_{k - 1}$, which will be the set $T$ from above.

\begin{remark}
    Throughout we will view the indices of the $T_i$ as elements of $\Zk$ and, in particular, all addition of indices is modulo $k$.
\end{remark}

Note that our definition of $T_i$ is slightly arbitrary. Whether we append or prepend words from $S$ to some $a \in \cF$, the change in the label of $a$ should always be the same. So, we could also have defined $T_i$ as the set $\set{a \in \cF \colon S \cap S^{i + 1}a \notempty}$. Fortunately, the following result tells us that these definitions are equivalent.

\begin{proposition}\label{prop:swap}
    For any positive integer $r$ and any $a \in \cF$, 
    \begin{equation*}
        S \cap S^r a \notempty \Leftrightarrow S \cap S^{r - 1} a S \notempty \Leftrightarrow \dotsb \Leftrightarrow S \cap S a S^{r - 1} \notempty \Leftrightarrow S \cap a S^r \notempty.
    \end{equation*}
\end{proposition}

\begin{proof}
    We first prove the case $r = 1$. Suppose that $S \cap Sa \notempty$. Then there is some $x$ such that $x, xa \in S$. Consider the sets $S, xS, x^2S, \dotsc, x^{k - 1}S, x^{k - 1}aS = x^{k - 2}(xa)S$. By \cref{lem:disjoint}, these cannot all be pairwise disjoint. Since $S$ is strongly $k$-product-free and $x \in S$, the sets $S, xS, \dotsc, x^{k - 1}S$ are pairwise disjoint. Since $S$ is strongly $k$-product-free and $xa \in S$, the sets $S, xS, \dotsc, x^{k - 2}S, x^{k - 2}(xa)S$ are pairwise disjoint. Thus $x^{k - 1}S$ and $x^{k - 1}aS$ are not disjoint and so $S \cap aS \notempty$.

    Let $f \colon \cF \to \cF$ be the \defn{reverse map} that reverses each word of $\cF$ (that is, reads them from right to left). The function $f$ is a measure-preserving involution. Let $\overline{S} = f(S)$. Now $\overline{S}$ is a strongly $k$-product-free subset of $\cF$ with $\db(\overline{S}) = \db(S) = 1/k$. In particular, the previous paragraph shows that $\overline{S} \cap \overline{S}a \notempty \Rightarrow \overline{S} \cap a \overline{S} \notempty$. Now, $\overline{S} \cap \overline{S}a = f(S \cap aS)$ and $\overline{S} \cap a \overline{S} = f(S \cap Sa)$ and so $S \cap aS \notempty \Rightarrow S \cap Sa \notempty$ concluding the case $r = 1$.

    For the general case it suffices to prove that for all non-negative integers $i, j$: $S \cap S^{i + 1} a S^j \notempty \Leftrightarrow S \cap S^i a S^{j + 1} \notempty$. Suppose that $S \cap S^{i + 1} a S^j \notempty$. Then there is $x_i \in S^i$ and $x_j \in S^j$ such that $S \cap S x_i a x_j \notempty$. Applying the $r = 1$ case to the word $x_i a x_j$ shows that $S \cap x_i a x_j S \notempty$ and so $S \cap S^i a S^{j + 1} \notempty$. The other direction is analogous.
\end{proof}

If the sets $T_1, \dots, T_{k-1}$ are supposed to correctly identify the labels of all words $a \in \cF$, then every $a$ should be in exactly one of these sets, and $S$ should satisfy $S \subset T_{k - 1}$. This is proved by the following proposition.

\begin{proposition}\label{prop:partition}
    The sets $T_0$, $T_1$, \ldots, $T_{k - 1}$ partition $\cF$ and $S \subset T_{k - 1}$.
\end{proposition}

\begin{proof}
    Let $a \in \cF$ and $x \in S$. Consider the sets $S$, $aS$, $axS$, $ax^2S$, \ldots, $ax^{k - 1}S$. By \cref{lem:disjoint}, these cannot all be pairwise disjoint. Since $S$ is strongly $k$-product-free and $x \in S$, the sets $aS$, $axS$, \ldots, $ax^{k - 1}S$ are pairwise disjoint. Hence there is some $r \in \set{0, 1, \dotsc, k - 1}$ such that $S \cap ax^rS \notempty$ and so $S \cap aS^{r + 1} \notempty$. That is, $\cup_{r = 0}^{k - 1} T_r = \cF$.

    We next show that the $T_i$ are pairwise disjoint (and so partition $\cF$). Suppose that $a \in T_i \cap T_j$ where $0 \leq i < j \leq k - 1$. Since $a \in T_i$, $S \cap SaS^i \notempty$ and so there is $x \in S$ and $y \in S^i$ such that $xay \in S$. Let
    \begin{equation*}
        C \coloneqq \set{s \in S \colon ays \in S} \subset S.
    \end{equation*}
    Consider the $k + 1$ sets 
    \begin{equation*}
        S, \quad xS, \quad x^2S, \quad \dotsc, \quad x^{k - 1}S, \quad x^{k - 1}ay (S \setminus C).
    \end{equation*}
    As $S$ is strongly $k$-product-free and $x \in S$, the first $k$ of these sets are pairwise disjoint. Similarly, noting that $x^{k - 1}ay = x^{k - 2}(xay)$ and $xay \in S$, we have that the last set is disjoint from each of the first $k - 1$. Finally, the last two sets are disjoint by the definition of $C$. Hence, all $k + 1$ sets are pairwise disjoint.

    Since $a \in T_j$ there are $z_1, \dotsc, z_{j + 1} \in S$ such that $z_1 \dotsb z_{j + 1}a \in S$. Consider the $k + 1$ sets 
    \begin{equation*}
        S, \quad z_1 S, \quad z_1^2 S, \quad \dotsc, \quad z_1^{k - j}S, \quad z_1^{k - j} z_2 S, \quad \dotsc, \quad z_1^{k - j} z_2 \dotsb z_j S, \quad z_1^{k - j} z_2 \dotsb z_{j + 1} a y C.
    \end{equation*}
    The first $k$ of these sets are pairwise disjoint as $S$ is strongly $k$-product-free. Similarly, noting that $z_1 z_2 \dotsb z_{j + 1} a \in S$, the last set is disjoint from each of $ S$, $z_1S$, \ldots, $z_1^{k - j - 1}S$. Now, by the definition of $C$, $ayC \subset S$. Using this and product-freeness shows that the last set is disjoint from each of $z_1^{k - j}S$, $z_1^{k - j} z_2 S$, \ldots, $z_1^{k - j} z_2 \dotsb z_j S$. Hence, all $k + 1$ sets are pairwise disjoint. This contradicts \cref{lem:disjointtwo} and so the $T_i$ do partition $\cF$.

    It remains to show that $S \subset T_{k - 1}$. Since $S$ is strongly $k$-product-free, for any $x \in S$, the set $S$ is disjoint from each of $xS$, $xS^2$, \ldots, $xS^{k - 1}$ and so $x \notin T_0 \cup \dotsb \cup T_{k - 2}$. Since the $T_i$ partition $\cF$, we must have $x \in T_{k - 1}$, as required.
\end{proof}

Given these two results, we already know that $a \in T_i$ should be labelled by $-i \in \Zk$. Next, we want to show that the label of a product $a b$ should be the sum of the labels of $a$ and $b$. We begin by proving that this is true whenever we append a word from $S$.

\begin{proposition}\label{prop:TS}
    The following hold for all $j \in \Zk$.
    \begin{enumerate}[label = \textnormal{(\alph{*})}]
        \item If $ax \in T_j$ and $x \in S$, then $a \in T_{j + 1}$. \label{prop:TSa}
        \item $T_{j + 1}S \subset T_j$. \label{prop:TSb}
    \end{enumerate}
\end{proposition}

\begin{proof}
    We first prove \ref{prop:TSa}. Suppose that $0 \leq j \leq k - 2$. We have $S \cap axS^{j + 1} \notempty$ and $x \in S$, so $S \cap a S^{j + 2} \notempty$ and so $a \in T_{j + 1}$. 
    
    Now suppose that $j = k - 1$. Consider the sets $S$, $aS$, $axS$, $ax^2S$, \ldots, $ax^{k - 1}S$. By \cref{lem:disjoint}, these cannot all be pairwise disjoint. Since $S$ is strongly $k$-product-free and $x \in S$, the sets $aS$, $axS$, \ldots, $ax^{k - 1}S$ are pairwise disjoint. Also, as $ax \in T_{k - 1}$ (and so $ax$ is not in $T_0 \cup T_1 \cup \dotsb \cup T_{k - 2}$ by \cref{prop:partition}), $S$ is disjoint from each of $axS$, $ax^2S$, \ldots, $ax^{k - 1}S$. Thus $S$ and $aS$ are not disjoint and so $a \in T_0$, as required.

    We now prove \ref{prop:TSb}. Let $a \in T_{j + 1}$ and $x \in S$. Suppose that $ax \in T_i$ (such an $i$ exists by \cref{prop:partition}). By \ref{prop:TSa}, $i + 1 = j + 1 \bmod{k}$ and so $i = j \bmod{k}$, as required.
\end{proof}

It is now an easy consequence that the labels of all $T_i$ are very well-behaved with respect to products.

\begin{proposition}\label{prop:TT}
    For all $i, j \in \Zk$, $T_i T_j \subset T_{i + j}$.
\end{proposition}

\begin{proof}
    Let $a \in T_i$ and $b \in T_j$. As $b \in T_j$ there are $x_1, x_2, \dotsc, x_{j + 1} \in S$ such that $b x_1 x_2 \dotsb x_{j + 1} \in S$. By \cref{prop:TS}\ref{prop:TSb},
    \begin{equation*}
        abx_1 \dotsb x_{j + 1} = a(bx_1 \dotsb x_{j + 1}) \in T_{i - 1}.
    \end{equation*}
    Applying \cref{prop:TS}\ref{prop:TSa} $j + 1$ times, once to remove each $x_{\ell}$, gives $ab \in T_{i - 1 + (j + 1)} = T_{i + j}$, as required.
\end{proof}

Finally, this allows us to complete the proof of \cref{thm:skpfstructure}.

\begin{proof}[Proof of \cref{thm:skpfstructure}]
    For each letter $\alpha \in \cA$, there is, by \cref{prop:partition}, a unique $i \in \Zk$ such that $\alpha \in T_i$. Label $\alpha$ with $i$. By \cref{prop:TT}, for each $i$,
    \begin{equation*}
        T_i = \set{w \in \cF \colon \text{the sum of the labels of letters in $w$ is } i \bmod{k}}.
    \end{equation*}
    In particular, by \cref{prop:partition},
    \begin{equation*}
        S \subset T_{k - 1} = \set{w \in \cF \colon \text{the sum of the labels of letters in $w$ is } -1 \bmod{k}}.
    \end{equation*}
    Note that $T_{k - 1}$ is strongly $k$-product-free: if $w$ is the concatenation of $\ell$ words from $T_{k - 1}$, then the sum of the labels of letters in $w$ is $-\ell \bmod{k}$.
    
    To obtain the result given in the statement of \cref{thm:skpfstructure} (i.e.\ with $1 \bmod{k}$ instead of $-1 \bmod{k}$) simply multiply the label of each letter by $-1$.
\end{proof}

\section{Structure of \texorpdfstring{$k$}{k}-product-free sets}\label{sec:kpfstructure}

In this section we prove \cref{thm:kpfstructure} assuming \cref{thm:pfdensity}. Let $k \geq 2$ be an integer with $k \notin \set{3, 5, 7, 13}$, let $\rho = \rho(k)$, and let $S \subset \cF$ be $k$-product-free satisfying $\db(S) = 1/\rho$. Note, by \cref{thm:pfdensity}, that $\di(S) = 1/\rho = \dsup(S)$ and so we may and will frequently apply \cref{lem:disjoint} with $n = \rho + 1$.

We will show that, in fact, $S$ is strongly $\rho$-product-free and so the result follows from \cref{thm:skpfstructure}. To this end we make the following definition.

\begin{definition}\label{def:SA}
    For a set $A \subset \bZ^+$, the set \defn{$S_A$} $\subset \cF$ is
    \begin{equation*}
        S_A \coloneqq \bigcap_{i \in A} S^i,
    \end{equation*}
    where we will omit set parentheses so, for example, $S_1 = S$ and $S_{1, 3} = S \cap S^3$.
\end{definition}

Since $S$ is $k$-product-free, $S_{1, k} = \emptyset$. It is enough for us to show that $S_{1, 2} = S_{1, 3} = \dotsb = S_{1, \rho} = \emptyset$ as then $S$ is strongly $\rho$-product-free. Note that the case $k = 2$ is immediate and so we assume that $k \geq 3$ from now on.

We need a quick technical lemma about the size of $\rho$.

\begin{lemma}\label{lem:rhosize}
    Let $k \geq 3$ be an integer with $k \notin \set{3, 5, 7, 13}$ and let $\rho = \rho(k)$. Then
    \begin{equation}\label{eq:rho}
        k - 1 \geq \max \set{(\rho - t)t(t + 1) \colon t \in \set{1, 2, \dotsc, \rho - 1}}.
    \end{equation}
\end{lemma}

\begin{proof}
    By the  arithmetic mean-geometric mean inequality, for any $t \in [0, \rho]$,
    \begin{equation*}
        (\rho - t) t (t + 1) = 4 (\rho - t)\tfrac{t}{2} \tfrac{t + 1}{2} \leq 4 \bigl(\tfrac{\rho + 1/2}{3}\bigr)^3 = 4/27 \cdot (\rho + 1/2)^3.
    \end{equation*}
    On the other hand, Lev~\cite[Lem.~18]{Lev03} proved that, for all positive integers $k \geq 2$,
    \begin{equation*}
        \rho(k) \leq 2 \log_2 k + 2.
    \end{equation*}
    Now, for all $k \geq 2400$,
    \begin{equation*}
        k - 1 \geq 4/27 \cdot (2 \log_2 k + 5/2)^3,
    \end{equation*}
    and so \eqref{eq:rho} holds. Now, if $\rho \geq 10$, then $k - 1 \geq 5 \times 7 \times 8 \times 9 = 2520$ and so \eqref{eq:rho} holds. We are left to check the remaining cases.
    \begin{itemize}[noitemsep]
        \item If $\rho = 2$, then the right-hand side of \eqref{eq:rho} is 2. The smallest $k \geq 3$ with $\rho = 2$ is 4.
        \item If $\rho = 3$, then the right-hand side of \eqref{eq:rho} is 6. The only $k \leq 6$ with $\rho = 3$ are 3 and 5.
        \item If $\rho = 4$, then the right-hand side of \eqref{eq:rho} is 12. The only $k \leq 12$ with $\rho = 4$ is 7.
        \item If $\rho = 5$, then the right-hand side of \eqref{eq:rho} is 24. The only $k \leq 24$ with $\rho = 5$ is 13.
        \item $\rho$ is always the power of a prime so there are no $k$ with $\rho = 6$.
        \item If $\rho = 7$, then the right-hand side of \eqref{eq:rho} is 60. The smallest $k$ with $\rho = 7$ is $61$.
        \item If $\rho = 8$, then the right-hand side of \eqref{eq:rho} is 90. The smallest $k$ with $\rho = 8$ is $421$.
        \item If $\rho = 9$, then the right-hand side of \eqref{eq:rho} is 126. The smallest $k$ with $\rho = 9$ is 841.\qedhere
    \end{itemize}
\end{proof}

We first show that $S_{1, \rho}$ is empty.

\begin{proposition}\label{prop:S1rempty}
    $S_{1, \rho} = \emptyset$.
\end{proposition}

\begin{proof}
    Suppose that $S_{1, \rho} \notempty$ and let $t \in \bZ^+$ be maximal with $S_{1, \rho, 2 \rho - 1, \dotsc, t(\rho - 1) + 1} \notempty$ where the indices form an arithmetic progression with common difference $\rho - 1$. Such a $t$ must exist as $k \equiv 1 \bmod{\rho - 1}$ and $S_{1, k} = \emptyset$. Let $w \in S_{1, \rho, 2 \rho - 1, \dotsc, t(\rho - 1) + 1}$.

    Taking $t = \rho - 1$ inside the maximum in \eqref{eq:rho}, we have $k - 1 \geq \rho(\rho - 1)$. We split into two cases based on the size of $k$.

    First suppose that $k - 1 > 2 \rho(\rho - 1)$. Let $\alpha \in \bZ^+$ be minimal such that $(\alpha - 1) \rho (\rho - 1) \geq k - 1$. Note that $\alpha \geq 4$. Write $\rho = 2a + b$ where $a = \floor{\rho /2}$ and $b \in \set{0, 1}$. Consider the following sets
    \begin{align*}
        & S, \quad w^{\rho - 1} S, \qquad w^{\alpha(\rho - 1)} S, \quad w^{(\alpha + 1)(\rho - 1)} S, \qquad w^{2 \alpha(\rho - 1)}, \quad w^{(2 \alpha + 1)(\rho - 1)} S \quad \dotsc, \\
        & w^{(a - 1) \alpha(\rho - 1)} S, \quad w^{((a - 1)\alpha + 1)(\rho - 1)}S, \qquad w^{a\alpha(\rho - 1)} S, \quad w^{(a \alpha + b)(\rho - 1)} S.
    \end{align*}
    We remark that these sets are formed by starting with $S$ and then alternating between prepending $w^{\rho - 1}$ and $w^{(\alpha - 1)(\rho - 1)}$. Two sets that differ only be a prepending of $w^{\rho - 1}$ are called a \defn{pair}: the pairs are the first and second sets; the third and fourth sets; \ldots. The number of sets listed is $2 a + b + 1 = \rho + 1$ and so these cannot all be pairwise disjoint by \cref{lem:disjoint}. 
    
    We first show that sets in different pairs are disjoint. If two such sets meet, then $S \cap w^{\ell(\rho - 1)}S \notempty$ for some integer $\ell$ satisfying $\alpha - 1 \leq \ell \leq a \alpha + b$. We will show that, for such an $\ell$, $w^{\ell(\rho - 1)} \in S^{k - 1}$ which contradicts $S_{1, k} = \emptyset$. Since $w \in S_{1, \rho, 2 \rho - 1, \dotsc, t(\rho - 1) + 1}$, we have $w^{\ell(\rho - 1)} \in S_{\ell(\rho - 1), (\ell + 1)(\rho - 1), \dotsc, \ell(\rho - 1)(t(\rho - 1) + 1)}$ where the indices form an arithmetic progression with common difference $\rho - 1$. It suffices to show that $k - 1$ is in this arithmetic progression. Since $k - 1$ is a multiple of $\rho - 1$, it is enough to show that $\ell(\rho - 1) \leq k - 1 \leq \ell(\rho - 1)(t(\rho - 1) + 1)$ for all integers $\ell$ satisfying $\alpha - 1 \leq \ell \leq a \alpha + b$. Now,
    \begin{equation*}
        \ell(\rho - 1)(t(\rho - 1) + 1) \geq \ell(\rho - 1) \rho \geq (\alpha - 1) \rho(\rho - 1) \geq k - 1
    \end{equation*}
    and
    \begin{align*}
        \ell(\rho - 1) & \leq (a \alpha + b)(\rho - 1) = (a \alpha + \rho - 2a)(\rho - 1) \\
        & = a(\alpha - 2)(\rho - 1) + \rho(\rho - 1) \leq \rho/2 \cdot (\alpha - 2)(\rho - 1) + \rho(\rho - 1) \\
        & = \alpha/2 \cdot \rho(\rho - 1) \leq (\alpha - 2) \rho(\rho - 1) < k - 1,
    \end{align*}
    where we used the minimality of $\alpha$ and the fact that $\alpha \geq 4$ in the final and penultimate inequality respectively.

    We second show that sets in the same pair are disjoint which gives the contradiction required to conclude the case $k - 1 > 2 \rho (\rho - 1)$. If two sets in the same pair are not disjoint, then $S \cap w^{\rho - 1}S \notempty$. But $w^{\rho - 1} \in S_{\rho - 1, 2(\rho - 1), \dotsc, (t(\rho - 1) + 1)(\rho - 1)}$ and so if $S \cap w^{\rho - 1}S \notempty$, then $S_{1, \rho, 2\rho - 1, \dotsc, (t(\rho - 1) + 1)(\rho - 1) + 1} \notempty$ which contradicts the maximality of $t$.

    Second suppose that $2\rho(\rho - 1) \geq k - 1 \geq \rho(\rho - 1)$. Consider the the following $\rho + 1$ sets
    \begin{equation*}
        S, \quad w^{\rho - 1}S, \quad w^{2(\rho - 1)}S, \quad \dotsc, \quad w^{\rho(\rho - 1)}S.
    \end{equation*}
    Since $t$ is maximal, consecutive sets are disjoint as in the previous case. If non-consecutive sets are not disjoint, then $S \cap w^{\ell(\rho - 1)}S \notempty$ for some integer $\ell$ satisfying $2 \leq \ell \leq \rho$. As before, $w^{\ell(\rho - 1)} \in S_{\ell(\rho - 1), (\ell + 1)(\rho - 1), \dotsc, \ell(\rho - 1)(t(\rho - 1) + 1)}$ and so it suffices to show that $\ell(\rho - 1) \leq k - 1 \leq \ell(\rho - 1)(t(\rho - 1) + 1)$ for such $\ell$. This is the case as
    \begin{equation*}
        \ell(\rho - 1)(t(\rho - 1) + 1) \geq \ell(\rho - 1)\rho \geq 2 \rho (\rho - 1) \geq k - 1
    \end{equation*}
    and
    \begin{equation*}
        \ell(\rho - 1) \leq \rho(\rho - 1) \leq k - 1.
    \end{equation*}
    Thus all $\rho + 1$ sets are disjoint contradicting \cref{lem:disjoint} and so $S_{1, \rho}$ is indeed empty.
\end{proof}

We now show that $S_{1, 2}$, \ldots, $S_{1, \rho - 1}$ are all empty.

\begin{proposition}\label{prop:S1dempty}
    For all $1 \leq d \leq \rho - 1$, $S_{1, d + 1} = \emptyset$.
\end{proposition}

\begin{proof}
    We argue via downwards induction on $d$ with the base case $d = \rho - 1$ given by \cref{prop:S1rempty}. Let $1 \leq d \leq \rho - 2$ be largest with $S_{1, d + 1} \notempty$ and let $t \in \bZ^+$ be maximal with $S_{1, d + 1, 2d + 1, \dotsc, td + 1} \notempty$. Such a $t$ exists as $k \equiv 1 \bmod{d}$. Let $w \in S_{1, d + 1, 2d + 1, \dotsc, td + 1}$.
    
    By the definition of $\rho$, both $d$ and $d + 1$ divide $k - 1$. Since $d$ and $d + 1$ are coprime we may write $k - 1 = \alpha d(d + 1)$ for some positive integer $\alpha$. Let $s \in \bZ^+$ be largest such that $ds \leq \rho - 1$. Write $\rho = a(s + 1) + b$ where $a = \floor{\rho/(s + 1)}$ and $b \in \set{0, 1, \dotsc, s}$. Consider the following $\rho + 1$ sets
    \begin{align*}
        & S, \quad w^d S, \quad \dotsc, \quad w^{sd} S, \qquad \qquad \qquad w^{(\alpha + s)d} S, \quad w^{(\alpha + s)d + d} S, \quad \dotsc, \quad w^{(\alpha + s)d + sd} S, \\
        & w^{2(\alpha + s)d} S, \quad \dotsc, \quad w^{2(\alpha + s)d + sd} S, \qquad \dotsc, \qquad w^{a(\alpha + s)d} S, \quad \dotsc, \quad w^{a(\alpha + s)d + bd} S.
    \end{align*}
    We remark that these sets are formed by starting with $S$, then prepending $w^d$ $s$ times, prepending $w^{\alpha d}$, then prepending $w^d$ $s$ times, prepending $w^{\alpha d}$, and so on. We group up the sets: the $1^{\text{st}}$ through $d^{\text{th}}$ sets are in the first group; the $(d + 1)^{\text{th}}$ through $(2d)^{\text{th}}$ sets are in the second group; and so on. 
    
    We first show that sets in different groups are disjoint. If two such sets meet, then $S \cap w^{\ell d} S \notempty$ for some integer $\ell$ satisfying $\alpha \leq \ell \leq a(\alpha + s) + b$. Now $w^{\ell d} \in S_{\ell d, (\ell + 1) d, \dotsc, \ell d(td + 1)}$ and so, since $k - 1$ is a multiple of $d$, it suffices to show that $\ell d \leq k - 1 \leq \ell d(td + 1)$ for all such $\ell$. Firstly,
    \begin{equation*}
        \ell d(td + 1) \geq \alpha d(d + 1) = k - 1.
    \end{equation*}
    Now,
    \begin{align*}
        \ell d & \leq (a(\alpha + s) + b) d = (a(\alpha + s) + \rho - a(s + 1))d \\
        & = (a(\alpha - 1) + \rho)d
    \end{align*}
    and we wish to show this is at most $k - 1 = \alpha d (d + 1)$ and so it is enough to show that $a(\alpha - 1) + \rho \leq \alpha(d + 1)$. By the maximality of $s$, $d(s + 1) \ge \rho$ and so $d \ge \rho/(s + 1) \geq a$. Hence, it suffices to show that $d(\alpha - 1) + \rho \leq \alpha(d + 1)$, or equivalently $\rho \leq \alpha + d$. But, by \cref{lem:rhosize},
    \begin{align*}
        \alpha d (d + 1) & = k - 1 \geq \max\set{(\rho - t)t(t + 1) \colon t \in \set{1, 2, \dotsc, \rho - 1}} \\
        & \geq (\rho - d) d (d + 1),
    \end{align*}
    and so we do indeed have $\rho \leq \alpha + d$.

    Next we show that sets in the same group are disjoint. If two consecutive sets in the same group meet, then $S \cap w^d S \notempty$. But $w^d \in S_{d, 2d, \dotsc, (td + 1)d}$ and so if $S \cap w^d S \notempty$, then $S_{1, d + 1, 2d + 1, \dotsc, (td + 1)d + 1} \notempty$ which contradicts the maximality of $t$. If two non-consecutive sets in the same group meet, then $S \cap w^{\ell d}S \notempty$ for some integer $\ell$ with $2 \leq \ell \leq s$. But $w^{\ell d} \in S_{\ell d}$ and so $S_{1, \ell d + 1} \notempty$. However, $d < 2d \leq \ell d \leq ds \leq \rho - 1$ and so this contradicts the maximality of $d$. 

    Hence, all $\rho + 1$ sets are pairwise disjoint which contradicts \cref{lem:disjoint}, as required.
\end{proof}

\Cref{prop:S1rempty,prop:S1dempty} together show that $S$ is strongly $\rho$-product-free. \Cref{thm:kpfstructure} then follows from \cref{thm:skpfstructure}.

\section{Steeplechases}\label{sec:steeple}

In this section, we develop some results which will be used in the next section to bound the density of a (strongly) $k$-product-free set $S$ and so prove \cref{thm:pfdensity}. To motivate our approach, assume that $S$ is strongly 3-product-free. To bound the density of $S$, we might hope that $\di(S) = \di(S^2) = \di(S^3)$. Because all of these sets are disjoint, this would imply that $\di(S \cup S^2 \cup S^3) = 3 \cdot \di(S)$ and so $\di(S) \le 1/3$, as required.

If $S$ is evenly distributed, such an argument works. Indeed, note that for all $w \in S$ we have $\di_{w\cF}(S^2) \ge \di_{w\cF}(wS) = \di(S)$, so the relative density of $S^2$ in $S \cF$ is at least $\di(S)$. If $S \cF$ covers all of $\cF$, this implies that $\di(S^2) \ge \di(S)$, and so $\di(S \cup S^2) \ge 2 \cdot \di(S)$. To include $S^3$ in the union, we can just repeat the argument. For $w \in S$ we have $\di_{w\cF}(S^2 \cup S^3) \ge \di_{w\cF}(w (S \cup S^2)) = \di(S \cup S^2) \ge 2 \cdot \di(S)$, giving $\di(S^2 \cup S^3) \ge 2 \cdot \di(S)$ and thus $\di(S \cup S^2 \cup S^3) \ge 3 \cdot \di(S)$, as required.

If $S$ is not evenly distributed, we want to ignore the part of $\cF$ where $S$ has a very low density. In the rest, the density of $S$ should be at least $\di(S)$ and $S$ should be somewhat evenly distributed. Within this part, we then want to show that $S \cup S^2$ has density $2 \cdot \di(S)$ and $S \cup S^2 \cup S^3$ has density $3 \cdot \di(S)$ to again obtain the sought result.

While the density of $S \cup S^2$ could be computed as before, this no longer works for $S \cup S^2 \cup S^3$. We only know that $S \cup S^2$ has a high density within a part of $\cF$, for example $\di_{v\cF}(S \cup S^2) \ge 2 \cdot \di(S)$ for some $v \in \cF$. This does not suffice to get a lower bound on $\di_{w\cF}(S^2 \cup S^3)$ in the calculation above.

Instead, note that $\di_{wv\cF}(S^2 \cup S^3) \ge \di_{wv\cF}(w (S \cup S^2)) = \di_{v\cF}(S \cup S^2) \ge 2 \cdot \di(S)$ which tells us that the relative density of $S^2 \cup S^3$ in $S v \cF$ is at least $2 \cdot \di(S)$. If we could now show that $S v \cF$ covers essentially all of $S \cF$, this would imply that $S^2 \cup S^3$ has density at least $2 \cdot \di(S)$ in $S \cF$ which in turn would suffice to show that $\di(S) \le 1/3$.

The technical arguments in this section are mostly devoted to showing that this is true, at least up to some small error. The idea is that we partition $S$ into prefix-free sets $(C_k)$ such that $C_{k+1} \subset C_k \cF$. At some point, the measure of $C_k$ will no longer drop. This means that $C_{k+1} \cF$ covers almost all subtrees of $C_k \cF$.

Now, $C_k v \cF$ will cover a fraction of size $\abs{\cA}^{-\abs{v}}$ of $C_k \cF$. We also know that all uncovered subtrees are covered by $C_{k+1} \cF$. So, $C_{k+1} v \cF$ will cover a fraction of size $\abs{\cA}^{-\abs{v}}$ of the still uncovered subtrees of $C_k \cF$, and the remaining subtrees are covered by $C_{k+2} \cF$. By repeating this argument with $C_{k+2} v \cF, C_{k+3} v \cF, \dots$, we can eventually cover almost all of $C_k \cF$ with $\bigcup_{\ell \ge k} C_\ell v \cF$. By deleting the first few layers of our partition of $S$, we therefore get that $S \cF$ is covered by $S v \cF$, which is what we need.

This motivate the following definition.

\begin{definition}[steeplechase]
    An infinite sequence $(C_k)$ of subsets of $\cF$ is a \defn{steeplechase} if, for each positive integer $k$,
    \begin{itemize}
        \item each $C_k$ is prefix-free and finite,
        \item every word in $C_{k + 1}$ has a proper prefix in $C_k$ (in particular, $C_{k + 1} \cF \subset C_k \cF$).
    \end{itemize}
    Steeplechase $(C_k)$ is \defn{spread} if $\max C_k < \min C_{k + 1}$ for all $k$ and is \defn{$\epsilon$-tight} if, for all $m, n$, $\abs{\mu(C_m) - \mu(C_n)} \leq \epsilon$.
\end{definition}

Every steeplechase contains a spread steeplechase. Indeed, note that $\min C_k \geq k$, since every word in $C_k$ has a proper prefix in $C_{k - 1}$. Let $\ell_1 = \max C_1$. Then $\min C_{\ell_1 + 1} > \ell_1 = \max C_1$. Let $\ell_2 = \max C_{\ell_1 + 1}$. Then $\min C_{\ell_2 + 1} > \max C_{\ell_1 + 1}$. Iteratively doing this gives a spread steeplechase $C_1, C_{\ell_1 + 1}, C_{\ell_2 + 1}, \dotsc$.

Since $C_k$ is prefix-free, $\mu(C_k) \in [0, 1]$. Also, for each $k$, $C_{k + 1} \cF \subset C_k \cF$ and so the sequence $(\mu(C_k))$ is non-increasing. In particular, this sequence tends to a limit. Hence the sequence is Cauchy: for any $\epsilon > 0$, there is a $K$ such that, for all $\ell, k \geq K$, $\abs{\mu(C_k) - \mu(C_\ell)} \leq \epsilon$. Thus, ignoring the first few $C_k$ gives an $\epsilon$-tight steeplechase.

In particular, given any steeplechase $(C_k)$ we may, by passing to a subsequence, assume that $(C_k)$ is both spread and $\epsilon$-tight.

The following lemma shows that, for any set $B \subset \cF$, there is a steeplechase that captures almost all of $B$.

\begin{lemma}\label{lem:steepleB}
    Let $\epsilon > 0$ and $B \subset \cF$. There is an $\epsilon$-tight spread steeplechase $(C_k)$ such that
    \begin{itemize}
        \item $C_1 \cup C_2 \cup \dotsb \subset B$,
        \item for all $k$ and all large $n$ \textup{(}in terms of $k$\textup{)}, $\mu((B \setminus C_k \cF)(n)) \leq \epsilon$,
        \item for all $k$, $\mu(C_k) \geq \di(B) - \epsilon$.
\end{itemize}
\end{lemma}

\begin{proof}
    For $x \in B$, let the \defn{headcount} of $x$ be 
    \begin{equation*}
        h(x) = \abs{\set{b \in B \colon b \text{ is a prefix of } x}}.
    \end{equation*}
    For each positive integer $k$, let $D_k = \set{x \in B \colon h(x) = k}$. Note that each $D_k$ is prefix-free and so $\mu(D_k) \leq 1$ for all $k$. Iteratively do the following procedure for each positive integer $k$.
    \begin{enumerate}
        \item Let $\ell_k$ be such that $\mu(D_k(\set{\ell_k + 1, \ell_k + 2, \dotsc})) \leq \epsilon/2^k$.
        \item Let $C_k = D_k(\set{1, 2, \dotsc, \ell_k})$.
        \item Remove $(D_k \setminus C_k)\cF$ from $B$ (including from all later $D_i$).
    \end{enumerate}
    Let $B'$ be the set remaining at the end of this procedure. Note that in step 3 the headcounts of words either remain the same or those words are removed from $B$ entirely. In particular, every word in $C_k$ has a proper prefix in $C_{k - 1}$. Also, by construction, $C_k$ is a finite subset of $B$. Thus $(C_k)$ is a steeplechase and $C_1 \cup C_2 \cup \dotsb \subset B$.

    Fix $k$ and let $n > \max\set{\ell_1, \dotsc, \ell_k}$. Any word of length $n$ in $B'$ is not in $C_1 \cup \dotsb \cup C_k$ and so has headcount greater than $k$ and so is in $C_k \cF$. Thus, $B'(n) \subset C_k\cF(n)$. Next note that $B'$ is obtained from $B$ by deleting all the $(D_t \setminus C_t) \cF$ and so,
    \begin{equation*}
        \mu((B \setminus C_k \cF)(n)) \leq \mu((B \setminus B')(n)) \leq \sum_t \mu((D_t \setminus C_t)(n)) \leq \sum_t \mu(D_t \setminus C_t) \leq \epsilon.
    \end{equation*}
    Finally, this implies that $\mu(B(n)) \leq \mu(C_k\cF(n)) + \epsilon \leq \mu(C_k) + \epsilon$. Averaging this over $n \in I_j$ and taking $j \to \infty$ gives $\di(B) \leq \mu(C_k) + \epsilon$.
    
    Hence $(C_k)$ is a steeplechase satisfying all three conditions. As noted above, we may, by passing to a subsequence, assume that $(C_k)$ is spread and $\epsilon$-tight. Passing to a subsequence does not affect the three conditions.
\end{proof}

We call the steeplechase $(C_k)$ given by \cref{lem:steepleB} an \defn{$\epsilon$-capturing steeplechase for $B$}.

\begin{lemma}\label{lem:steepleCw}
    Let $(C_k)$ be an $\epsilon$-tight spread steeplechase. For every $w \in \cF$ there is an $N$ such that the following holds. If $C = C_1 \cup C_2 \cup \dotsb \cup C_N$ and $n \geq \max C_N + \abs{w}$, then
    \begin{equation*}
        \mu((C_1 \cF \setminus Cw\cF)(n)) \leq 2 \epsilon.
    \end{equation*}
\end{lemma}

\begin{proof}
    Let $N$ sufficiently large in terms of $\abs{w}$ and $\abs{\cA}$ and let $n \geq \max C_N + \abs{w}$. Now
    \begin{equation*}
        \mu((C_1 \cF \setminus C_N\cF)(n)) = \mu((C_1 \cF)(n)) - \mu((C_N \cF)(n)) = \mu(C_1) - \mu(C_N) \leq \epsilon,
    \end{equation*}
    since $C_N \cF \subset C_1 \cF$ and $(C_k)$ is $\epsilon$-tight. Hence, it suffices to prove that
    \begin{equation*}
        \mu((C_N \cF \setminus Cw \cF)(n)) \leq \epsilon.
    \end{equation*}
    Let $X$ be the following finite prefix-free set
    \begin{equation*}
        X = \set{s \in C_N \colon s \text{ has no prefix in } Cw}.
    \end{equation*}
    Note that $(C_N\cF \setminus Cw \cF)(n) \subset (X\cF)(n)$ and so
    \begin{equation*}
        \mu((C_N\cF \setminus Cw \cF)(n)) \leq \mu((X \cF)(n)) = \mu(X).
    \end{equation*}
    Recall the random infinite word $\W = \alpha_1 \alpha_2 \dotsb$ and corresponding random walk defined in \cref{sec:density}. Since $X$ is prefix-free, $\mu(X) = \bP(\W \text{ hits } X)$ and it suffices to show this probability is at most $\epsilon$. Let $K$ be the largest integer with $1 + K \abs{w} \leq N - \abs{w}$. If $\W$ hits $X$, then $\W$ hits $C_N$ and so, since $(C_k)$ is a steeplechase, $\W$ hits each of $C_1$, $C_{1 + \abs{w}}$, \ldots, $C_{1 + K \abs{w}}$. Also, $\W$ must avoid each of $C_1 w$, $C_{1 + \abs{w}} w$, \ldots, $C_{1 + K \abs{w}} w$ in order to hit $X$.

    We reveal the letters of $\W$ one-by-one. We wait until $\W$ hits/avoids $C_1$ (this will certainly be known by the time the length of $\W$ is $\max C_1$). If $\W$ avoids $C_1$, then $\W$ avoids $X$. If $\W$ hits $C_1$, then we reveal the next $\abs{w}$ letters of $\W$ and check if they spell $w$ (this has probability $\abs{\cA}^{-\abs{w}}$). If they do, then $\W$ avoids $X$. If they do not, then we wait until $\W$ hits/avoids $C_{1 + \abs{w}}$: note that this has not already happened since $(C_k)$ is spread and so $\min C_{1 + \abs{w}} \geq \max C_1 + \abs{w}$. If $\W$ avoids $C_{1 + \abs{w}}$, then $\W$ avoids $X$. If $\W$ hits $C_{1 + \abs{w}}$, then we reveal the next $\abs{w}$ letters of $\W$ and check if they spell $w$ (this has probability $\abs{A}^{-\abs{w}}$). We continue this procedure with the final check being whether the next $\abs{w}$ letters of $\W$ after it hits $C_{1 + K\abs{w}}$ spell $w$. Note that each check has probability $\abs{\cA}^{-\abs{w}}$ and is independent of the previous checks (new letters are involved in each check). If $\W$ hits $X$, then $\W$ must fail each of these spelling checks and so the probability that $\W$ hits $X$ is at most
    \begin{equation*}
        (1 - \abs{\cA}^{-\abs{w}})^{K + 1}.
    \end{equation*}
    By taking $N$ (and so $K$) sufficiently large in terms of $\abs{w}$ and $\abs{\cA}$ we may ensure this is at most $\epsilon$, as required.
\end{proof}

Before proving our key technical result for our density proofs (\cref{lem:densityAB}) we will need to define the relative density of $B$ on $C\cF$. If $C \subset \cF$ is finite and interval $I$ satisfies $\min I \geq \max C$, then the \defn{relative density of $B$ in $C\cF$ on interval $I$} is
\begin{equation*}
    d^I_{C \cF}(B) \coloneqq \frac{\mu(B(I) \cap C\cF)}{\mu(\cF(I) \cap C\cF)}.
\end{equation*}
Suppose $C$ is also prefix-free. Then, by \cref{obs:densityCF}, $\mu(\cF(I) \cap C\cF) = \abs{I} \mu(C)$. Also $(c\cF \colon c \in C)$ partition $C \cF$. In particular,
\begin{align*}
    d^I_{C\cF}(B) & = \abs{I}^{-1} \mu(C)^{-1} \sum_{n \in I} \mu(B(n) \cap C \cF)\\
    & = \sum_{c \in C}  \abs{I}^{-1} \mu(C)^{-1} \sum_{n \in I} \mu(B(n) \cap c\cF) \\
    & = \sum_{c \in C} \frac{\mu(c)}{\mu(C)} \cdot d^I_{c\cF}(B) \\
    & = \mu(C)^{-1} \sum_{c \in C} d^I(B \cap c \cF) \\
    & = \mu(C)^{-1} \cdot d^I(B \cap C \cF).
\end{align*}
For every set $B$ that we consider in this paper and every word $c$, the sequence $d^{I_j}_{c\cF}(B)$ converges (to $\di_{c\cF}(B)$). Thus the sequence $d^{I_j}_{C\cF}(B)$ converges to a limit \defn{$\di_{C \cF}(B)$}. Again, these limits are additive.

\begin{observation}\label{obs:densityCFavg}
    Let $C \subset \cF$ be finite and prefix-free. Then
    \begin{equation*}
        \di_{C\cF}(B) = \sum_{c \in C} \frac{\mu(c)}{\mu(C)} \cdot \di_{c\cF}(B) = \mu(C)^{-1} \cdot \di(B \cap C \cF).
    \end{equation*}
\end{observation}

Now for the key technical lemma for our density results.

\begin{lemma}\label{lem:densityAB}
    Let $\epsilon > 0$ and let $A, B \subset \cF$. If $(C_k)$ is an $\epsilon$-capturing steeplechase for $A$ with $\mu(C_1) \geq 2 \epsilon + \epsilon^{1/3}$, then
    \begin{equation*}
        \di_{C_1 \cF}(AB) \geq \dsup(B) - 3 \epsilon^{1/3}.
    \end{equation*}
\end{lemma}

\begin{proof}
    Let $w \in \cF$ be such that
    \begin{equation*}
        \di_{w \cF}(B) \geq \dsup(B) - \epsilon^{1/3}.
    \end{equation*}
    Apply \cref{lem:steepleCw} to $(C_k)$ and $w$ to give an $N$ such that letting $C = C_1 \cup \dotsb \cup C_N$, if $n \geq \max C_N + \abs{w}$, then
    \begin{equation*}
        \mu((C_1 \cF \setminus Cw \cF)(n)) \leq 2 \epsilon.
    \end{equation*}
    We may greedily choose $\widetilde{C} \subset C$ (starting with shorter words first) such that $\widetilde{C}w$ is prefix-free and $\widetilde{C}w \cF = Cw \cF$. Note that
    \begin{align*}
        2 \epsilon & \geq \mu((C_1 \cF \setminus \widetilde{C}w \cF)(n)) \geq \mu((C_1 \cF)(n)) - \mu((\widetilde{C}w \cF)(n)) \\
        & = \mu(C_1) - \mu(\widetilde{C}w) \geq 2 \epsilon + \epsilon^{1/3} - \mu(\widetilde{C}w)
    \end{align*}
    and so $\mu(\widetilde{C}w) \geq \epsilon^{1/3}$. 
    
    Let $I$ be an interval with $\min I \geq \max C_N + \abs{w}$ and let $X \subset \cF$. Note that $\widetilde{C}w \cF \subset C_1 \cF$ and so
    \begin{align*}
        d^I_{C_1 \cF}(X) & = \abs{I}^{-1} \mu(C_1)^{-1} \sum_{n \in I} \mu(X(n) \cap C_1 \cF) \\
        & \geq \abs{I}^{-1} \mu(C_1)^{-1} \sum_{n \in I} \mu(X(n) \cap \widetilde{C} w \cF) \\
        & \geq \abs{I}^{-1} \sum_{n \in I} \frac{\mu(X(n) \cap \widetilde{C} w \cF)}{\mu(\widetilde{C}w) + 2 \epsilon}.
    \end{align*}
    Using the fact that $x/(y + 2 \epsilon) \geq x/y - 2 \epsilon x/y^2 \geq x/y - 2 \epsilon^{1/3}$ for $\epsilon > 0$, $x \in [0, 1]$, and $y \geq \epsilon^{1/3}$, we have
    \begin{equation*}
        d^I_{C_1 \cF}(X) \geq d^I_{\widetilde{C}w\cF}(X) - 2 \epsilon^{1/3}.
    \end{equation*}
    Setting $X = AB$, $I = I_j$, and taking $j$ to infinity gives
    \begin{equation}\label{eq:lem:densityAB:1}
        \di_{C_1 \cF}(AB) \geq \di_{\widetilde{C}w \cF}(AB) - 2 \epsilon^{1/3}.
    \end{equation}
    Now,
    \begin{align*}
        \di_{\widetilde{C}w \cF}(AB) & = \sum_{c \in \widetilde{C}} \frac{\mu(c)}{\mu(\widetilde{C})} \cdot \di_{c w \cF}(AB) \\
        & \geq \sum_{c \in \widetilde{C}} \frac{\mu(c)}{\mu(\widetilde{C})} \cdot \di_{c w \cF}(cB) \\
        & = \sum_{c \in \widetilde{C}} \frac{\mu(c)}{\mu(\widetilde{C})} \cdot \di_{w \cF}(B) \\
        & = \di_{w \cF}(B) \geq \dsup(B) - \epsilon^{1/3},
    \end{align*}
    where the first equality used \cref{obs:densityCFavg}, the first inequality used the fact that $c \in \widetilde{C} \subset C \subset A$, the second equality used \cref{lem:densitycancel}, and the second inequality is due to the choice of $w$. Combining this with \eqref{eq:lem:densityAB:1} gives the required result.
\end{proof}

\section{Density of (strongly) \texorpdfstring{$k$}{k}-product-free sets}\label{sec:pfdensity}

In this section we prove \cref{thm:pfdensity}, making use of the machinery developed in the previous section. Part~\ref{thm:skpfdensity} has a simple iterating proof which uses that a strongly $k$-product-free $S$ is disjoint from each of $S^2$, $S^3$, \ldots, $S^k$.

\begin{proof}[Proof of \cref{thm:pfdensity}\ref{thm:skpfdensity}]
    Let $S \subset \cF$ be strongly $k$-product-free and let $\epsilon > 0$ be sufficiently small. Let $(C_k)$ be an $\epsilon$-capturing steeplechase for $S$, as given by \cref{lem:steepleB}. If $\mu(C_1) < 2 \epsilon + \epsilon^{1/3}$, then $\db(S) = \di(S) \leq \mu(C_1) + \epsilon < 3 \epsilon + \epsilon^{1/3}$ which is less than $1/k$. Otherwise, by \cref{lem:densityAB}, $\di_{C_1 \cF}(S^2) \geq \dsup(S) - 3 \epsilon^{1/3}$. Since $S$ is strongly $k$-product-free, $S$ and $S^2$ are disjoint and so
    \begin{equation*}
        \di_{C_1 \cF}(S \cup S^2) \geq \di_{C_1 \cF}(S) + \dsup(S) - 3 \epsilon^{1/3}.
    \end{equation*}
    Now, by \cref{obs:densityCFavg},
    \begin{equation*}
        \dsup(S \cup S^2) \geq \di_{C_1 \cF}(S \cup S^2) \geq \di_{C_1 \cF}(S) + \dsup(S) - 3 \epsilon^{1/3},
    \end{equation*}
    and so, by \cref{lem:densityAB},
    \begin{equation*}
        \di_{C_1 \cF}(S^2 \cup S^3) \geq \dsup(S \cup S^2) - 3 \epsilon^{1/3} \geq \di_{C_1 \cF}(S) + \dsup(S) - 6 \epsilon^{1/3}.
    \end{equation*}
    Hence,
    \begin{equation*}
        \dsup(S \cup S^2 \cup S^3) \geq \di_{C_1 \cF}(S \cup S^2 \cup S^3) \geq 2\di_{C_1 \cF}(S) + \dsup(S) - 6 \epsilon^{1/3}.
    \end{equation*}
    Iterating this argument gives
    \begin{equation*}
        1 \geq \dsup(S \cup S^2 \cup \dotsb \cup S^k) \geq (k - 1) \di_{C_1 \cF}(S) + \dsup(S) - 3(k - 1) \epsilon^{1/3}.
    \end{equation*}
    But, since  $(C_k)$ is $\epsilon$-capturing for $S$, $\di_{C_1 \cF}(S) = \mu(C_1)^{-1} \cdot \di(S \cap C_1 \cF) \geq \di(S \cap C_1 \cF) \geq \di(S) - \epsilon$. Hence, $(k - 1) \di(S) + \dsup(S) \leq 1 + (k - 1) \epsilon + 3(k - 1) \epsilon^{1/3}$. As $\epsilon$ is arbitrarily small,
    \begin{equation*}
        1 \geq (k - 1) \di(S) + \dsup(S).
    \end{equation*}
    But $\dsup(S) \geq \di(S)$ and so $\db(S) = \di(S) \leq 1/k$. Furthermore, if $\db(S) = 1/k$, then $\di(S) = 1/k$ and so $\dsup(S) \leq 1/k$, as required.
\end{proof}

The argument for part~\ref{thm:kpfdensity} ($k$-product-free sets) is more involved. It is not necessary to keep track of the error term depending on $\epsilon$ (as we eventually take $\epsilon$ to zero). We introduce some notation to simplify the argument. Write \defn{$x \lesssim y$} to mean that $x \leq y + f(\epsilon)$ where the error term $f(\epsilon)$ depends only on $k$ and $\epsilon$ and goes to zero as $\epsilon$ goes to zero (in all cases $f(\epsilon)$ will be a polynomial in $\epsilon^{1/3}$).

To improve clarity and motivate the proof we first sketch a proof of \cref{thm:pfdensity}\ref{thm:kpfdensity} for $k = 3$. For full details see the proof of \cref{prop:Asequence} that follows.

\begin{proof}[Proof of \cref{thm:pfdensity}\ref{thm:kpfdensity} for $k = 3$]
    Let $S \subset \cF$ be 3-product-free and let $\epsilon > 0$ be sufficiently small. Let $(C_k)$ be an $\epsilon$-capturing steeplechase for $S_{1, 2} = S \cap S^2$ (recall \cref{def:SA}), as given by \cref{lem:steepleB}. Since $S$ is 3-product-free, $S_{1, 2}$ is strongly 3-product-free.
    
    We claim that $\di(S \cap C_1 \cF) \lesssim 1/3 \cdot \mu(C_1)$. If $\mu(C_1) < 2 \epsilon + \epsilon^{1/3}$, then this is immediate. Otherwise $\mu(C_1) \geq 2 \epsilon + \epsilon^{1/3}$ and so, by \cref{lem:densityAB}, $\di_{C_1 \cF}(S_{1, 2} S) \gtrsim \dsup(S)$. Since $S$ is 3-product-free, $S$ and $S_{1, 2} S$ are disjoint and so
    \begin{equation*}
        \dsup(S \cup S_{1, 2}S) \geq \di_{C_1 \cF}(S \cup S_{1, 2}S) \gtrsim \di_{C_1 \cF}(S) + \dsup(S).
    \end{equation*}
    Then, by \cref{lem:densityAB},
    \begin{equation*}
        \di_{C_1 \cF}(S_{1, 2} S \cup S_{1, 2}^2 S) = \di_{C_1 \cF}(S_{1, 2}(S \cup S_{1, 2}S)) \gtrsim \di_{C_1 \cF}(S) + \dsup(S).
    \end{equation*}
    Since $S$ is 3-product-free, $S$ is disjoint from $S_{1, 2}S \cup S_{1, 2}^2 S$ and so
    \begin{equation*}
        1 \geq \dsup(S \cup S_{1, 2}S \cup S_{1, 2}^2 S) \geq \di_{C_1 \cF}(S \cup S_{1, 2} S \cup S_{1, 2}^2 S) \gtrsim 2 \di_{C_1 \cF}(S) + \dsup(S).
    \end{equation*}
    But $\dsup(S) \geq \di_{C_1 \cF}(S)$ and so $\di_{C_1 \cF}(S) \lesssim 1/3$. \Cref{obs:densityCFavg} then gives $\di(S \cap C_1 \cF) \lesssim 1/3 \cdot \mu(C_1)$, as claimed.

    We have bounded the density of $S$ on the part of $\cF$ where $S_{1, 2}$ is dense. We now bound the density of $S$ on the rest. Let $S' = S \setminus (S_{1, 2} \cup C_1 \cF)$ and $(D_k)$ be an $\epsilon$-capturing steeplechase for $S'$, as given by \cref{lem:steepleB}. By passing to a subsequence we may and will assume that $\min D_1 > \max C_1$. Since $S$ is 3-product-free and $S' \cap S^2 = \emptyset$, $S'$ is strongly 3-product-free.
    
    We claim that $\di(S' \cap D_1 \cF) \lesssim 1/3 \cdot \mu(D_1)$. If $\mu(D_1) < 2 \epsilon + \epsilon^{1/3}$, then this is immediate. Otherwise, by \cref{lem:densityAB}, $\di_{D_1 \cF} (S'S) \gtrsim \dsup(S)$. Now $S'$ and $S'S$ are disjoint since $S' \cap S^2 = \emptyset$. Thus
    \begin{equation*}
        \dsup(S' \cup S'S) \geq \di_{D_1 \cF}(S' \cup S'S) \gtrsim \di_{D_1 \cF}(S') + \dsup(S).
    \end{equation*}
    Then, by \cref{lem:densityAB},
    \begin{equation*}
        \di_{D_1 \cF}((S')^2 \cup (S')^2 S) = \di_{D_1 \cF}(S'(S' \cup S'S)) \gtrsim \di_{D_1 \cF}(S') + \dsup(S).
    \end{equation*}
    Since $S$ is 3-product-free and $S' \cap S^2 = \emptyset$, $S'$ is disjoint from $(S')^2 \cup (S')^2 S$ and so
    \begin{equation*}
        1 \geq \dsup(S' \cup (S')^2 \cup (S')^2 S) \geq \di_{D_1 \cF}(S' \cup (S')^2 \cup (S')^2 S) \gtrsim 2 \di_{D_1 \cF}(S') + \dsup(S).
    \end{equation*}
    But $\dsup(S) \geq \dsup(S') \geq \di_{D_1 \cF}(S')$ and so $\di_{D_1 \cF}(S') \lesssim 1/3$. \Cref{obs:densityCFavg} then gives $\di(S' \cap D_1 \cF) \lesssim 1/3 \cdot \mu(D_1)$, as claimed.
    
    By the definition of $S'$ and since $\min D_1 > \max C_1$, it follows that $C_1 \cF$ and $D_1 \cF$ are disjoint (see proof of \cref{prop:Asequence} for more details). In particular, $C_1$ and $D_1$ are disjoint and their union is prefix-free. Hence $\mu(C_1) + \mu(D_1) \leq 1$. Thus,
    \begin{equation*}
        \di(S \cap C_1 \cF) + \di(S' \cap D_1 \cF) \lesssim 1/3 \cdot (\mu(C_1) + \mu(D_1)) \leq 1/3.
    \end{equation*}
    Since $(C_k)$ and $(D_k)$ are $\epsilon$-capturing, it follows (see the proof of \cref{claim:densityoutsteeple} below) that very little of $S$ lies outside $(S \cap C_1 \cF) \cup (S' \cap D_1 \cF)$. In particular, $\di(S) \lesssim 1/3$. Since $\epsilon$ can be arbitrarily small, we have $\db(S) = \di(S) \leq 1/3$. For the moreover part see the proof of \cref{prop:Asequence} below.
\end{proof}

For general $k$ the argument is a more involved version of the above. We first consider some $S_{A_1}$, take some $\epsilon$-capturing steeplechase, $(C_k^{(1)})$ for $S_{A_1}$ and show the density of $S$ relative to $C_1 \cF$ is at most $1/\rho$. We then repeat this step for some $S_{A_2}$, $S_{A_3}$, \ldots. In future steps we may use the fact that we have dealt with previous $S_{A_i}$. \Cref{prop:Asequence} says that if we have chosen a suitable sequence $A_1, A_2, \dotsc$, then we obtain the required bound on $\db(S)$, and \cref{prop:Asequenceexist} shows that for each $k$ there is a suitable sequence of $A_i$. These combine to complete the proof of \cref{thm:pfdensity}\ref{thm:kpfdensity}.

Note in the statement below that \defn{$dA_\ell$} is the sumset
\begin{equation*}
    dA_\ell \coloneqq \set{a_1 + \dotsb + a_d \colon a_1, \dotsc, a_d \in A_{\ell}}.
\end{equation*}

\begin{proposition}\label{prop:Asequence}
	Let $k \geq 2$ be an integer and $A_1, \dots, A_m \subset \bN$ be a sequence of sets with $A_m = \set{1}$. Suppose that for all $\ell \in [m]$ there exist positive integers $d_1, \dotsc, d_{\rho - 1}$ such that, for all $1 \leq i \leq j \leq \rho - 1$, either
	\begin{itemize}
		\item $k \in \set{1} \cup (1 + (d_i + d_{i + 1} + \dotsb + d_j) A_\ell)$ or
		\item $A_{\ell'} \subset \set{1} \cup (1 + (d_i + d_{i + 1} + \dotsb + d_j) A_\ell)$ for some $1 \le \ell' < \ell$.
	\end{itemize}
	If $S \subset \cF$ is $k$-product-free, then $\db(S) \le 1/\rho$. Moreover, if $\db(S) = 1/\rho$, then $\dsup(S) = 1/\rho$.
\end{proposition}

\begin{proof}
	Let $\epsilon > 0$ be sufficiently small. We define the following sets and steeplechases. Take $S^{(1)} \coloneqq S$, $R^{(1)} \coloneqq S_{A_1} \cap S^{(1)}$, and let $(C_k^{(1)})$ be an $\epsilon$-capturing steeplechase for $R^{(1)}$, as given by \cref{lem:steepleB}. For $\ell = 2, 3, \dotsc, m$, iteratively do the following:
	\begin{itemize}
		\item Set $S^{(\ell)} \coloneqq S \setminus (S_{A_1} \cup \dotsb \cup S_{A_{\ell - 1}} \cup (C_1^{(1)} \cup \dotsb \cup C_1^{(\ell - 1)})\cF)$ and $R^{(\ell)} \coloneqq S_{A_\ell} \cap S^{(\ell)}$.
		\item Take $(C_k^{(\ell)})$ to be an $\epsilon$-capturing steeplechase for $R^{(\ell)}$. By passing to a subsequence of the steeplechase we may and will assume that $\min C_1^{(\ell)} > \max C_1^{(\ell - 1)}$.
	\end{itemize}
	
	\begin{claim}\label{claim:densityinsteeple}
		For each $\ell \in [m]$, $\di(S^{(\ell)} \cap C_1^{(\ell)}\cF) \lesssim 1/\rho \cdot \mu(C_1^{(\ell)})$.
	\end{claim}
	
	\begin{proof}
	   Firstly, if $\mu(C_1^{(\ell)}) < 2\epsilon + \epsilon^{1/3}$, then
        \begin{equation*}
            \di(S^{(\ell)} \cap C_1^{(\ell)}\cF) \leq \di(C_1^{(\ell)} \cF) = \mu(C_1^{(\ell)}) \lesssim 1/\rho \cdot \mu(C_1^{(\ell)}).
        \end{equation*}
        Hence, we may assume from now on that $\mu(C_1^{(\ell)}) \geq 2\epsilon + \epsilon^{1/3}$. Consider the sets $S^{(\ell)}$ and $(R^{(\ell)})^{d_{\rho - 1}} S$. Note that $(R^{(\ell)})^{d_{\rho - 1}} \subset S_{d_{\rho - 1} A_{\ell}}$. Hence, if $S^{(\ell)}$ and $(R^{(\ell)})^{d_{\rho - 1}} S$ meet, then $S^{(\ell)} \cap S_{1 + d_{\rho - 1} A_{\ell}} \notempty$. By the proposition statement, this implies that $S^{(\ell)} \cap S_k \notempty$ or $S^{(\ell)} \cap S_{A_{\ell'}} \notempty$ (for $\ell' < \ell$). $k$-product-freeness rules out the former and the definition of $S^{(\ell)}$ the latter. Therefore, $S^{(\ell)}$ and $(R^{(\ell)})^{d_{\rho - 1}} S$ are disjoint and so,
		\begin{align*}
			\dsup(S^{(\ell)} \cup (R^{(\ell)})^{d_{\rho - 1}} S) & \geq \di_{C_1^{(\ell)}\cF}(S^{(\ell)} \cup (R^{(\ell)})^{d_{\rho - 1}} S) \\
			& = \di_{C_1^{(\ell)}\cF}(S^{(\ell)}) + \di_{C_1^{(\ell)}\cF}((R^{(\ell)})^{d_{\rho - 1}} S).
		\end{align*}
		Note that $(C_k^{(\ell)})$ is an $\epsilon$-capturing steeplechase for $R^{(\ell)}$ and so, by \cref{lem:densityAB},
		\begin{equation*}
			\dsup(S^{(\ell)} \cup (R^{(\ell)})^{d_{\rho - 1}} S) \gtrsim \di_{C_1^{(\ell)}\cF}(S^{(\ell)}) + \dsup(S).
		\end{equation*}
		Iterating this procedure, exactly as in the proofs of \cref{thm:pfdensity}\ref{thm:skpfdensity} and the $k = 3$ case above, gives
		\begin{equation}\label{eq:densityC1sup}
			\begin{split}
				1 \geq \dsup(S^{(\ell)} \cup (R^{(\ell)})^{d_1} S^{(\ell)} & \cup \dotsb \cup (R^{(\ell)})^{d_1 + \dotsb + d_{\rho - 2}} S^{(\ell)} \cup (R^{(\ell)})^{d_1 + \dotsb + d_{\rho - 1}}S) \\
				&\gtrsim (\rho - 1) \di_{C_1^{(\ell)}\cF}(S^{(\ell)}) + \dsup(S).
			\end{split}
		\end{equation}
		Now, $\dsup(S) \geq \dsup(S^{(\ell)}) \geq \di_{C_1^{(\ell)}\cF}(S^{(\ell)})$ and so $\di_{C_1^{(\ell)}\cF}(S^{(\ell)}) \lesssim 1/\rho$. The claim follows from \cref{obs:densityCFavg}.
	\end{proof}
	
	We next show that very little of $S$ has not been captured by the previous claim.
	
	\begin{claim}\label{claim:densityoutsteeple}
		For all large $n$, $\mu(S(n) \setminus \bigcup_{\ell} (S^{(\ell)} \cap C_1^{(\ell)} \cF)) \leq m \epsilon$.
	\end{claim}
	
	\begin{proof}
		For each $\ell$, $(C_k^{(\ell)})$ is an $\epsilon$-capturing steeplechase for $R^{(\ell)}$ and so, for all large $n$,
		\begin{equation*}
			\mu(R^{(\ell)}(n) \setminus C_1^{(\ell)}\cF) \leq \epsilon.
		\end{equation*}
		Hence, it is enough to show that $S\setminus \bigcup_{\ell} (S^{(\ell)} \cap C_1^{(\ell)} \cF)) \subset \bigcup_{\ell} (R^{(\ell)}\setminus C_1^{(\ell)}\cF)$. Fix $w \in S \setminus \bigcup_{\ell} (S^{(\ell)} \cap C_1^{(\ell)} \cF))$. Let $\ell$ be maximal with $w \in S^{(\ell)}$ (such an $\ell$ exists as $S^{(1)} = S$). Since $w \in S \setminus \bigcup_{\ell} (S^{(\ell)} \cap C_1^{(\ell)} \cF))$, we have $w \notin C_1^{(\ell)} \cF$. We claim that $w \in S_{A_\ell}$. If $\ell = m$, then this is immediate ($S_{A_m} = S_1 = S$). If $\ell < m$, then, by the maximality of $\ell$, we must have $w \in S_{A_\ell} \cup C_1^{(\ell)}\cF$ and so $w \in S_{A_\ell}$. Thus, $w \in (S_{A_\ell} \cap S^{(\ell)}) \setminus C_1^{(\ell)} \cF = R^{(\ell)} \setminus C_1^{(\ell)} \cF$, as required.
	\end{proof}
	
	We now note that $C_1^{(1)} \cF$, \ldots, $C_1^{(m)} \cF$ are pairwise disjoint. If not then some $w_i \in C_1^{(i)}$ is a prefix of some $w_j \in C_1^{(j)}$ (for $i \neq j$). Now, by construction, $\min C_1^{(\ell)} > \max C_1^{(\ell - 1)}$ for all $\ell$ and so $i < j$. On the other hand, $w_j \in C_1^{(j)} \subset R^{(j)} \subset S^{(j)}$ and so $w_j \notin C_1^{(i)} \cF$, a contradiction. In particular, $C_1^{(\ell)}$, \ldots, $C_1^{(m)}$ are pairwise disjoint and their union is prefix-free.
	
	We can now show that $\db(S) \leq 1/\rho$. Summing \cref{claim:densityinsteeple} over $\ell$ gives
    \begin{equation}\label{eq:densitySCrho}
		\di(\bigcup_\ell (S^{(\ell)} \cap C_1^{(\ell)}\cF)) \lesssim 1/\rho \cdot \mu(C_1^{(1)} \cup \dotsb \cup C_1^{(m)}) \leq 1/\rho.
	\end{equation}
	Then, by \cref{claim:densityoutsteeple}, we obtain $\di(S) \lesssim 1/\rho$. Noting that $\epsilon$ can be arbitrarily small we have $\db(S) = \di(S) \leq 1/\rho$.
	
	Finally suppose that $\di(S) = \db(S) = 1/\rho$. We must have `equality' in \cref{claim:densityinsteeple} and \eqref{eq:densitySCrho}. That is, $\mu(C_1^{(1)} \cup \dotsb \cup C_1^{(m)}) \gtrsim 1$ and $\di(S^{(\ell)} \cap C_1^{(\ell)} \cF) \gtrsim 1/\rho \cdot \mu(C_1^{(\ell)})$ for all $\ell \in [m]$. Take $\ell$ with $\mu(C_1^{(\ell)}) \geq 1/(2m)$. Then, by \cref{obs:densityCFavg}, $\di_{C_1^{(\ell)}\cF}(S^{(\ell)}) \gtrsim 1/\rho$. But then \eqref{eq:densityC1sup} gives $\dsup(S) \lesssim 1/\rho$. Since $\epsilon$ can be arbitrarily small, we have $\dsup(S) \leq 1/\rho$, as required.
\end{proof}

We now show that there is always a sequence of sets satisfying \cref{prop:Asequence}. The sequence chosen here is motivated by the proofs of \cref{prop:S1rempty,prop:S1dempty}.

\begin{proposition}\label{prop:Asequenceexist}
    For every integer $k \geq 2$ there is a sequence $A_1, \dotsc, A_m$ satisfying the hypothesis of \cref{prop:Asequence}.
\end{proposition}

\begin{proof}
    We deal with the cases $k = 2, 3, 5, 7, 13$ first.
    \begin{itemize}
        \item $k = 2$: take $A_1 = \set{1}$ (with $d_1 = 1$),
        \item $k = 3$: take $A_1 = \set{1, 2}$ (with $d_1 = d_2 = 1$) and $A_2 = \set{1}$ (with $d_1 = d_2 = 1$),
        \item $k = 5$: take $A_1 = \set{1, 3}$ (with $d_1 = d_2 = 2$) and $A_2 = \set{1}$ (with $d_1 = d_2 = 2$),
        \item $k = 7$: take
        \begin{align*}
            & A_1 = \set{1, 3} &&(\text{with } d_1 = d_2 = d_3 = 2),\\
            & A_2 = \set{1, 2} &&(\text{with } d_1 = d_2 = d_3 = 1),\\
            & A_3 = \set{1, 4} &&(\text{with } d_1 = d_2 = d_3 = 1), \text{ and}\\
            & A_4 = \set{1} &&(\text{with } d_1 = d_2 = d_3 = 1),
        \end{align*}
        \item $k = 13$: take
        \begin{align*}
            & A_1 = \set{1, 4} &&(\text{with } d_1 = d_2 = d_3 = d_4 = 3),\\
            & A_2 = \set{1, 2} &&(\text{with } d_1 = d_2 = d_3 = d_4 = 3),\\
            & A_3 = \set{1, 3, 5, 7} &&(\text{with } d_1 = d_2 = d_3 = d_4 = 1),\\
            & A_4 = \set{1, 3} &&(\text{with } d_1 = d_2 = d_3 = d_4 = 1),\\
            & A_5 = \set{1, 5} &&(\text{with } d_1 = d_2 = d_3 = d_4 = 1), \text{ and}\\
            & A_6 = \set{1} &&(\text{with } d_1 = d_2 = d_3 = d_4 = 1).
        \end{align*}
    \end{itemize}
    We now turn to $k \notin \set{2, 3, 5, 7, 13}$. For positive integers $d$ and $t$, let $B_{d,t} \coloneqq  \set{1,d+1,2d+1,\dotsc,td+1}$. We construct $A_1, A_2, \dotsc$ by taking all the sets $B_{d,t}$ for $1 \le d \le \rho-1$ and $1 \le t < (k-1)/d$ in the order of decreasing $d$ and then decreasing $t$, and add the set $\set{1}$ to the end.

    Consider a set $A_\ell = B_{d,t}$. We need to show that $A_\ell$ satisfies \cref{prop:Asequence}. Let $s \in \bZ^+$ be maximal such that $d s \le \rho-1$, and $\alpha \in \bZ^+$ be minimal such that $\alpha d (d+1) \ge k-1$. For $1 \le i \le \rho-1$, define
    \begin{equation*}
        d_i = \begin{cases}
            \alpha d & \text{if } i \equiv 0 \bmod{s+1} \text{ and } \alpha \neq 2, \\
            d & \text{otherwise}.
        \end{cases}
    \end{equation*}
    For any $1 \le i \le j \le \rho-1$, we have that $d_i + d_{i+1} + \dots + d_j = \beta d$ for some integer $\beta$ satisfying $1 \le \beta \le \rho-1 + (\alpha-1) (\rho-1) / (s+1)$. Moreover, by definition of the $d_i$, either $\beta \ge \alpha$ or $\beta \le s$. Note that
    \begin{equation*}
        1 + (d_i + d_{i+1} + \dots + d_j) A_\ell = \set{1 + \beta d, 1 + (\beta + 1) d, \dots, 1 + \beta d (t d + 1)}.
    \end{equation*}
    If $\beta = 1$, then $B_{d,t+1} \subset \set{1} \cup (1 + (d_i + d_{i+1} + \dots + d_j) A_\ell)$. Now, either $k \in B_{d,t+1}$ or $B_{d,t+1} = A_{\ell'}$ for some $\ell' < \ell$ and so $A_\ell$ satisfies \cref{prop:Asequence}.

    If $1 < \beta \le s$, then $B_{\beta d,1} \subset \set{1} \cup (1 + (d_i + d_{i+1} + \dots + d_j) A_\ell)$. Since $d < \beta d \le d s \le \rho-1$, it holds that $B_{\beta d,1} = A_{\ell'}$ for some $\ell' < \ell$ and so $A_\ell$ satisfies \cref{prop:Asequence}.
    
    If $\beta \ge \alpha$, we claim that $k \in 1 + (d_i + d_{i+1} + \dots + d_j) A_\ell$. Since $k-1$ is a multiple of $d$, it suffices to show that $\beta d \le k - 1 \le \beta d (t d + 1)$. Firstly,
    \begin{equation*}
        \beta d (t d + 1) \ge \alpha d (d + 1) \ge k - 1.
    \end{equation*}
    For the second inequality, if $d \le \rho-2$, it holds that $\alpha d (d+1) = k-1$ as observed in the proof of \cref{prop:S1dempty}. Furthermore, we have
    \begin{equation*}
        \beta d \le (\rho + (\alpha-1) \rho / (s+1)) d \le (\rho + (\alpha-1) d) d.
    \end{equation*}
    where the second inequality follows from $d (s+1) \ge \rho$. This is less than $k-1 = \alpha d (d+1)$ if $d (\alpha-1) + \rho \le \alpha (d+1)$, which is true for $k \notin \set{2, 3, 5, 7, 13}$ as shown in the proof of \cref{prop:S1dempty}. On the other hand, if $d = \rho-1$, we have
    \begin{equation*}
        \beta d \le (d + (\alpha-1) d / (s+1)) d \le ((\alpha+1) / 2) d^2 \le ((\alpha+1) / 2) d (d+1).
    \end{equation*}
    If $\alpha \ge 3$, this is at most $(\alpha-1) d (d+1) < k-1$ as required. If $\alpha \le 2$, we can observe that $\beta \le \rho-1$ to obtain $\beta d \le \rho (\rho-1) \le k-1$ where the last inequality was shown in the proof of \cref{prop:S1rempty}. In all cases, we have $\beta d \le k-1$ as required. Hence, $k \in 1 + (d_i + d_{i+1} + \dots + d_j) A_\ell$, and so $A_\ell$ satisfies \cref{prop:Asequence}.

    Finally, for $A_\ell = \set{1}$, we can simply pick $d_1 = \dots = d_{\rho-1} = 1$. We then get that $B_{j-i+1,1} \subset \set{1} \cup (1 + (d_i + d_{i+1} + \dots + d_j) A_\ell)$. Since $1 \le j-i+1 \le \rho-1$, it holds that $B_{j-i+1,1} = A_{\ell'}$ for some $\ell' < \ell$ and so $A_\ell$ satisfies \cref{prop:Asequence}.
\end{proof}

\Cref{prop:Asequence,prop:Asequenceexist} combine to give \cref{thm:pfdensity}\ref{thm:kpfdensity} and so we have indeed proved \cref{thm:pfdensity} in this section, as promised.

\section{Product-free sets in the free group}\label{sec:pfdensityfree}

We now adapt our methods to the free group and prove \cref{thm:kpffree}. Throughout, $\fF$ denotes the free group on a finite alphabet $\cA$, and $S \subset \fF$ is a $k$-product-free set whose density we want to bound. We always assume that all words are in reduced form. Moreover, \defn{$A B$} denotes the product of two sets $A, B \subset \fF$ without cancellation, that is
\begin{equation*}
    A B \coloneqq \set{w \in \fF: \text{ there is a substring decomposition } w = a b \text{ with } a \in A \text{ and } b \in B}.
\end{equation*}
In particular, $C \fF$ consists of all words with a prefix in $C$. We equip $\fF$ with the measure \defn{$\mu$} defined as $\mu(w) = 1 / \abs{\fF(\abs{w})}$. If $\W = \alpha_1 \alpha_2 \dotsb$ is a random infinite word where each $\alpha_{i+1}$ is an independent uniformly random letter other than $\alpha_i^{-1}$, then
\begin{equation*}
    \mu(w) = \bP(\W \text{ hits } w) = \begin{cases}
        (2 \abs{\cA})^{-1} (2 \abs{\cA} - 1)^{-(\abs{w}-1)} & \text{if $w$ is not the empty word}, \\
        1 & \text{otherwise}.
    \end{cases}
\end{equation*}
As before, for $B \subset \fF$, $\mu(B) = \sum_{w \in B} \mu(w)$ is the expected number of times that $\W$ hits $B$. So, we can make the following observations corresponding to \cref{obs:measureC,obs:densityCF}.

\begin{observation}
    If $C \subset \fF$ is prefix-free, then $\mu(C) \leq 1$, $\mu(C \fF(n)) \leq \mu(C)$ for all $n$, and $\mu(C \fF(n)) = \mu(C)$ if $C$ is finite and $n \ge \max C$.
\end{observation}

We now define the relative density of subsets of $\fF$ as follows.

\begin{definition}[relative density]
    Let $B \subset \fF$, and $\fG$ be a subsemigroup of $\fF$ with $\fG(n) \neq \emptyset$ for all sufficiently large $n$. If $I$ is an interval, then the \defn{relative density of $B$ in $\fG$ on interval $I$} is
    \begin{equation*}
        d_\fG^I(B) \coloneqq \frac{\mu(B(I) \cap \fG)}{\mu(\fG(I))} = \mu(\fG(I))^{-1} \sum_{n \in I} \mu(B(n) \cap \fG).
    \end{equation*}
    If $\fG(I) = \emptyset$, we will take the relative density to be $0$ by convention, and $d^I(B) \coloneqq d_\fF^I(B)$.
\end{definition}

The upper Banach density of $B$ is then $\db(B) = \limsup_{I \to \infty} d^I(B)$. At this point, we can again diagonalise to obtain a sequence $(I_j)$ is such that, for every $\fG$ and $B$ that we consider in our proofs, $(d_\fG^{I_j}(B))$ converges to some limit \defn{$\di_\fG(B)$}, and $\di(S) = \db(S)$. These limits are again additive. Also note that $d_\fG^I(B) = d^I(B \cap \fG) / d^I(\fG)$ and so $\di_\fG(B) = \di(B \cap \fG) / \di(\fG)$ if $\di(\fG) > 0$. We define sup density as follows.

\begin{definition}[sup density]
    For a set $B$, the \defn{sup density of $B$ in $\fG$} is
    \begin{equation*}
        \dsupG(B) \coloneqq \sup_{w \in \fG} \di_{w \fF \cap \fG}(B).
    \end{equation*}
\end{definition}

From now on, let $\fG = \fF^{\alpha\beta} \subset \fF$ be the subsemigroup of $\fF$ consisting of all words starting with $\alpha$ and ending in $\beta$ where $\alpha, \beta \in \cA \cup \cA^{-1}$ and $\alpha \neq \beta^{-1}$. A random sequence argument shows the following.

\begin{observation}\label{obs:weightsubtreeg}
    Let $C \subset \fG$ be finite and prefix-free. Then, for all $n \ge \max C + 2$,
    \begin{equation*}
        \mu((C \fF \cap \fG)(n)) \geq \frac{\mu(C)}{(2 \abs{\cA} - 1)^2}.
    \end{equation*}
\end{observation}

In particular, $\di(w\fF \cap \fG) > 0$ for all $w \in \fG$. As in \cref{lem:densitycancel}, subtree densities of $G$ satisfy the property that we may strip away prefixes.

\begin{lemma}\label{lem:densitycancelfree}
    If $w, v \in \fG$, then $\di_{w v \fF \cap \fG}(w B) = \di_{v \fF \cap \fG}(B)$.
\end{lemma}

\begin{proof}
    For $u \in \fG$, note that $\mu(w u) = a \cdot \mu(u)$ where $a = (2 \abs{A} - 1)^{-\abs{w}}$. So, if $X \subset \fG$ is finite, then $\mu(w X) = a \cdot \mu(X)$. Let $I$ be any interval with $\min I > \abs{wv}$. Then
    \begin{equation*}
        d^I_{wv\fF \cap \fG}(wB) = \frac{\mu((wB)(I) \cap wv\fF)}{\mu((wv\fF \cap \fG)(I))} = \frac{\mu(B(I - \abs{w}) \cap v\fF)}{\mu((v\fF \cap \fG)(I - \abs{w}))}.
    \end{equation*}
    where we used that $wv\fF \cap \fG = w (v\fF \cap \fG)$. For any $X \subset \fG$, the fact that $\mu(X(n)) \in [0, 1]$ implies that
    \begin{equation*}
        \abs{\mu(X(I)) - \mu(X(I - \abs{w}))} = \abs[\bigg]{\sum_{n \in I} \mu(X(n)) - \sum_{n \in I - \abs{w}} \mu(X(n))} \leq \abs{w}.
    \end{equation*}
    Therefore,
    \begin{equation*}
        \frac{\mu(B(I) \cap v\fF) - \abs{w}}{\mu((v\fF \cap \fG)(I)) + \abs{w}} \leq d^I_{wv\fF \cap \fG}(wB) \leq \frac{\mu(B(I) \cap v\fF) + \abs{w}}{\mu((v\fF \cap \fG)(I)) - \abs{w}}.
    \end{equation*}
    Set $I = I_j$ and take $j$ to infinity. From $\di(v\fF \cap \fG) > 0$ it follows $\mu((v\fF \cap \fG)(I_j)) \to \infty$. Hence, both bounds above tend to $\di_{v\fF \cap \fG}(B)$ and so $\di_{w v \fF \cap \fG}(w B) = \di_{v \fF \cap \fG}(B)$.
\end{proof}

We can also obtain the following analogue to \cref{obs:densityCFavg}.

\begin{observation}\label{obs:densityCFavgfree}
    Let $C \subset \fG$ be finite and prefix-free. Then
    \begin{equation*}
        \di_{C\fF \cap \fG}(B) = \sum_{c \in C} \frac{\di(c\fF \cap \fG)}{\di(C\fF \cap \fG)} \cdot \di_{c\fF \cap \fG}(B) = \frac{\di(B \cap C\fF)}{\di(C\fF \cap \fG)}.
    \end{equation*}
\end{observation}

\begin{proof}
    Because $\di(c\fF \cap \fG) > 0$ for all $c \in C$, and therefore also $\di(C\fF \cap \fG) > 0$, it holds that
    \begin{equation*}
        \di_{c\fF \cap \fG}(B) = \frac{\di(B \cap c\fF)}{\di(c\fF \cap \fG)} \quad \text{ and } \quad \di_{C\fF \cap \fG}(B) = \frac{\di(B \cap C\fF)}{\di(C\fF \cap \fG)}.
    \end{equation*}
    Since $\di$ is additive, this implies that
    \begin{equation*}
        \di_{C\fF \cap \fG}(B) = \frac{\di(B \cap C\fF)}{\di(C\fF \cap \fG)} = \sum_{c \in C} \frac{\di(B \cap c\fF)}{\di(C\fF \cap \fG)} = \sum_{c \in C} \frac{\di(c\fF \cap \fG)}{\di(C\fF \cap \fG)} \cdot \di_{c\fF \cap \fG}(B). \qedhere
    \end{equation*}
\end{proof}

Steeplechases in the free group can be defined exactly as for the free semigroup.

\begin{definition}[steeplechase]
    An infinite sequence $(C_k)$ of subsets of $\fG$ is a \defn{steeplechase} if, for each positive integer $k$,
    \begin{itemize}
        \item each $C_k$ is prefix-free and finite,
        \item every word in $C_{k + 1}$ has a proper prefix in $C_k$ (in particular, $C_{k + 1} \fF \subset C_k \fF$).
    \end{itemize}
    Steeplechase $(C_k)$ is \defn{spread} if $\max C_k < \min C_{k + 1}$ for all $k$ and is \defn{$\epsilon$-tight} if, for all $m, n$, $\abs{\mu(C_m) - \mu(C_n)} \leq \epsilon$.
\end{definition}

The following lemma is an analogue to \cref{lem:steepleB}.

\begin{lemma}\label{lem:steepleBfree}
    Let $\epsilon > 0$ and $B \subset \fG$. There is an $\epsilon$-tight spread steeplechase $(C_k)$ such that
    \begin{itemize}
        \item $C_1 \cup C_2 \cup \dotsb \subset B$,
        \item for all $k$ and all large $n$ \textup{(}in terms of $k$\textup{)}, $\mu((B \setminus C_k \fF)(n)) \leq \epsilon$,
        \item for all $k$, $\mu(C_k) / \di(\fG) \geq \di_\fG(B) - \epsilon$.
    \end{itemize}
\end{lemma}

\begin{proof}
    This is very similar to the proof of \cref{lem:steepleB}. For each positive integer $k$, let $D_k = \set{x \in B \colon h(x) = k}$. Iteratively do the following for each positive integer $k$.
    \begin{enumerate}
        \item Let $\ell_k$ be such that $\mu(D_k(\set{\ell_k + 1, \ell_k + 2, \dotsc})) \leq \epsilon \cdot \di(\fG)/2^k$.
        \item Let $C_k = D_k(\set{1, 2, \dotsc, \ell_k})$.
        \item Remove $(D_k \setminus C_k)\fF$ from $B$ (including from all later $D_i$).
    \end{enumerate}
    Then $(C_k)$ is a steeplechase and $C_1 \cup C_2 \cup \dotsb \subset B$. Fix $k$ and let $n > \max\set{\ell_1, \dotsc, \ell_k}$. Then,
    \begin{equation*}
        \mu((B \setminus C_k \fF)(n)) \leq \sum_t \mu(((D_t \setminus C_t)\fF)(n)) \leq \sum_t \mu(D_t \setminus C_t) \leq \epsilon \cdot \di(\fG) \leq \epsilon.
    \end{equation*}
    Finally, this implies that $\mu(B(n)) \leq \mu(C_k\fF(n)) + \epsilon \cdot \di(\fG) = \mu(C_k) + \epsilon \cdot \di(\fG)$. Averaging this over $n \in I_j$ and taking $j \to \infty$ gives $\di(B) \leq \mu(C_k) + \epsilon \cdot \di(\fG)$ and therefore $\di_\fG(B) = \di(B) / \di(\fG) \leq \mu(C_k) / \di(\fG) + \epsilon$. By passing to a subsequence, we may assume that $(C_k)$ is spread and $\epsilon$-tight.
\end{proof}

We call the steeplechase $(C_k)$ given by \cref{lem:steepleBfree} an \defn{$\epsilon$-capturing steeplechase for $B$}. There is also an analogue to \cref{lem:steepleCw}.

\begin{lemma}\label{lem:steepleCwfree}
    Let $(C_k)$ be an $\epsilon$-tight spread steeplechase. For every $w \in \fG$ there is an $N$ such that the following holds. If $C = C_1 \cup C_2 \cup \dotsb \cup C_N$ and $n \geq \max C_N + \abs{w}$, then
    \begin{equation*}
        \mu(((C_1 \fF \cap \fG) \setminus (Cw \fF \cap \fG))(n)) \leq 2 \epsilon.
    \end{equation*}
\end{lemma}

\begin{proof}
    Note that $(C_1 \fF \cap \fG) \setminus (Cw \fF \cap \fG) = (C_1 \fF \setminus Cw \fF) \cap \fG \subset C_1 \fF \setminus Cw \fF$, and so it suffices to show that $\mu((C_1 \fF \setminus Cw\fF)(n)) \leq 2 \epsilon$.
    
    We can show this exactly as in the proof of \cref{lem:steepleCw}, we only need $\W$ to be the random walk from the beginning of this section. As a consequence, if $\W$ hits $C_i$, the probability that the next $\abs{w}$ letters of $\W$ spell $w$ is $(2 \abs{\cA} - 1)^{-\abs{w}}$. Importantly, this uses the fact that the last letter of a word in $C_i$ is $\beta$ and the first letter of $w$ is $\alpha$, and $\alpha \neq \beta^{-1}$.
\end{proof}

Now we can prove the key technical lemma for our density results, corresponding to \cref{lem:densityAB}.

\begin{lemma}\label{lem:densityABfree}
    Let $\epsilon > 0$ and let $A, B \subset \fG$. If $(C_k)$ is an $\epsilon$-capturing steeplechase for $A$ with $\mu(C_1) \geq (2 \abs{\cA} - 1)^2 (2 \epsilon + \epsilon^{1/3})$, then
    \begin{equation*}
        \di_{C_1 \fF \cap \fG}(AB) \geq \dsupG(B) - 3 \epsilon^{1/3}.
    \end{equation*}
\end{lemma}

\begin{proof}
    \Cref{obs:weightsubtreeg} implies that $\mu((C_1 \fF \cap \fG)(n)) \geq \mu(C_1) / (2 \abs{\cA} - 1)^2 \geq 2 \epsilon + \epsilon^{1/3}$ for $n \geq \max C_1 + 2$. We proceed as in the proof of \cref{lem:densityAB}. Let $w \in \fG$ be such that
    \begin{equation*}
        \di_{w \fF \cap \fG}(B) \geq \dsupG(B) - \epsilon^{1/3}.
    \end{equation*}
    Apply \cref{lem:steepleCwfree} to $(C_k)$ and $w$ to give an $N$ such that letting $C = C_1 \cup \dotsb \cup C_N$, if $n \geq \max C_N + \abs{w}$, then $\mu(((C_1 \fF \cap \fG) \setminus (Cw \fF \cap \fG))(n)) \leq 2 \epsilon$. We may greedily choose $\widetilde{C} \subset C$ such that $\widetilde{C}w$ is prefix-free and $\widetilde{C}w \fF = Cw \fF$. Note that
    \begin{align*}
        2 \epsilon & \geq \mu(((C_1 \fF \cap \fG) \setminus (\widetilde{C}w \fF \cap \fG))(n)) \\
        & \geq \mu((C_1 \fF \cap \fG)(n)) - \mu((\widetilde{C}w \fF \cap \fG)(n)) \\
        & \geq 2 \epsilon + \epsilon^{1/3} - \mu((\widetilde{C}w \fF \cap \fG)(n))
    \end{align*}
    and so $\mu((\widetilde{C}w \fF \cap \fG)(n)) \geq \epsilon^{1/3}$ as well as $\mu((C_1 \fF \cap \fG)(n)) \leq \mu((\widetilde{C}w \fF \cap \fG)(n)) + 2 \epsilon$.
    
    Let $I$ be an interval with $\min I \geq \max C_N + \abs{w}$, so $\mu((\widetilde{C}w \fF \cap \fG)(I)) \geq \abs{I} \epsilon^{1/3}$ and $\mu((C_1 \fF \cap \fG)(I)) \leq \mu((\widetilde{C}w \fF \cap \fG)(I)) + \abs{I} 2 \epsilon$. Let $X \subset \fG$. Note that $\widetilde{C}w \fF \subset C_1 \fF$ and so
    \begin{align*}
        d^I_{C_1 \fF \cap \fG}(X) & = \frac{\mu(X(I) \cap C_1 \fF)}{\mu((C_1 \fF \cap \fG)(I))} \geq \frac{\mu(X(I) \cap \widetilde{C} w \fF)}{\mu((C_1 \fF \cap \fG)(I))} \geq \frac{\mu(X(I) \cap \widetilde{C} w \fF)}{\mu((\widetilde{C}w \fF \cap \fG)(I)) + \abs{I} 2 \epsilon}.
    \end{align*}
    Using the fact that $x/(y + \abs{I} 2 \epsilon) \geq x/y - \abs{I} 2 \epsilon x/y^2 \geq x/y - 2 \epsilon^{1/3}$ for $\epsilon > 0$, $0 \leq x \leq \abs{I}$, and $y \geq \abs{I} \epsilon^{1/3}$, we have
    \begin{equation*}
        d^I_{C_1 \fF \cap \fG}(X) \geq d^I_{\widetilde{C}w\fF \cap \fG}(X) - 2 \epsilon^{1/3}.
    \end{equation*}
    Setting $X = AB$, $I = I_j$, and taking $j$ to infinity gives
    \begin{equation}\label{eq:lem:densityABfree:1}
        \di_{C_1 \fF \cap \fG}(AB) \geq \di_{\widetilde{C}w \fF \cap \fG}(AB) - 2 \epsilon^{1/3}.
    \end{equation}
    Now,
    \begin{align*}
        \di_{\widetilde{C}w \fF \cap \fG}(AB) & = \sum_{c \in \widetilde{C}} \frac{\di(cw\fF \cap \fG)}{\di(\widetilde{C}w\fF \cap \fG)} \cdot \di_{c w \fF \cap \fG}(AB) \\
        & \geq \sum_{c \in \widetilde{C}} \frac{\di(cw\fF \cap \fG)}{\di(\widetilde{C}w\fF \cap \fG)} \cdot \di_{c w \fF \cap \fG}(cB) \\
        & = \sum_{c \in \widetilde{C}} \frac{\di(cw\fF \cap \fG)}{\di(\widetilde{C}w\fF \cap \fG)} \cdot \di_{w \fF \cap \fG}(B) \\
        & = \di_{w \fF \cap \fG}(B) \geq \dsupG(B) - \epsilon^{1/3},
    \end{align*}
    where the first equality used \cref{obs:densityCFavgfree}, the first inequality used the fact that $c \in \widetilde{C} \subset C \subset A$, the second equality used \cref{lem:densitycancelfree}, and the second inequality is due to the choice of $w$. Combining this with \eqref{eq:lem:densityABfree:1} gives the required result.
\end{proof}

At this point, we have recovered all important technical results that we needed to bound the density of (strongly) $k$-product-free sets in the free semigroup. We can now simply use exactly the same arguments as in \cref{sec:pfdensity}. We only need to replace $C_i \cF$ by $C_i \fF \cap \fG$, $\mu(C_i)$ by $\di(C_i \fF \cap G)$, $\dsup$ by $\dsupG$, and all references by references to the corresponding results in this section to prove the following analogue of \cref{thm:pfdensity}.

\begin{theorem}\label{thm:pfdensityg}
    Let $k \geq 2$ be an integer, $\cA$ be a finite set, $\fF$ be the free group with alphabet $\cA$, and $\fG = \fF^{\alpha\beta}$ be the subsemigroup of $\fF$ consisting of all words starting with $\alpha$ and ending with $\beta$ where $\alpha, \beta \in \cA \cup \cA^{-1}$ and $\alpha \neq \beta^{-1}$.
    \begin{enumerate}[label = {\textup{(\alph{*})}}]
        \item If $S \subset \fG$ is strongly $k$-product-free, then $\di_\fG(S) \leq 1/k$. Moreover, if $\di_\fG(S) = 1/k$, then $\dsupG(S) = 1/k$. \label{thm:skpfdensityg}
        \item If $S \subset \fG$ is $k$-product-free, then $\di_\fG(S) \leq 1/\rho(k)$. Moreover, if $\di_\fG(S) = 1/\rho(k)$, then $\dsupG(S) = 1/\rho(k)$. \label{thm:kpfdensityg}
    \end{enumerate}
\end{theorem}

The arguments from Ortega, Ru\'{e}, and Serra \cite{ORS23} show that a density bound on (strongly) $k$-product-free sets in $\fF^{\alpha\beta}$ immediately translates to a density bound in $\fF$. Therefore, \cref{thm:ORS,thm:kpffree} are immediate corollaries of \cref{thm:pfdensityg}.

\section{Open problems}\label{sec:open}

A first natural problem left open is to determine the structure of the extremal $k$-product-free sets for $k \in \set{3, 5, 7, 13}$. For $k = 5, 7, 13$, we conjecture that the extremal sets are exactly as in \cref{thm:kpfstructure}. The extremal sets for $k = 3$ will be slightly more complicated. Indeed, while $1 + 3 \ZNN$ and $2 + 3\ZNN$ are both maximal 3-sum-free subsets of the non-negative integers of density $1/3$, so are both $\set{1, 2} + 6 \ZNN$ and $\set{5, 6} + 6 \ZNN$ (\L uczak and Schoen~\cite{LS97} showed that there are no others). We conjecture the corresponding result holds for the free semigroup.

\begin{conjecture}
    Let $\cA$ be a finite set and $\cF$ be the free semigroup with alphabet $\cA$. If $S \subset \cF$ is 3-product-free and $\db(S) = 1/3$, then one of the following hold. Either it is possible to label each letter of $\cA$ with a label in $\bZ/3\bZ$ such that $S$ is a subset of
    \begin{equation*}
        \set{w \in \cF \colon \text{the sum of the labels of letters in $w$ is $1 \bmod{3}$}},
    \end{equation*}
    or it is possible to label each letter of $\cA$ with a label in $\bZ/6\bZ$ such that $S$ is a subset of
    \begin{equation*}
        \set{w \in \cF \colon \text{the sum of the labels of letters in $w$ is $1, 2 \bmod{6}$}}.
    \end{equation*}
\end{conjecture}

\L uczak~\cite{Luczak95} proved that every sum-free subset of the non-negative integers with density greater than $2/5$ is a subset of the odd integers (\L uczak and Schoen proved similar results for (strongly) $k$-sum-free sets). Such strengthenings for subsets of the free semigroup are false as the constants $1/k$ in \cref{thm:skpfstructure} and $1/\rho(k)$ in \cref{thm:kpfstructure} cannot be replaced by anything smaller. For example, let $k = 2$, $T$ be the set of words of odd length, and $x$ be any word of even length. Let
\begin{equation*}
    \begin{split}
        T' \coloneqq & \set{w \in T \colon \text{neither $x$ nor $w$ is a prefix or suffix of the other}} \\
        & \cup \set{xwx \colon \text{$xwx$ has length $1 \bmod{3}$}}.
    \end{split}
\end{equation*}
Then $T'$ is product-free, has density at least $1/2 - 2 \abs{\cA}^{-\abs{x}}$, and is not a subset of an odd-occurrence set (in fact, a set of positive density would need to be removed before this happens). Nonetheless, $T'$ is a small perturbation from  the odd-occurrence set $T$. Hence, it is natural to ask whether there is some form of stability.

\begin{conjecture}
    For each $\delta > 0$, is there some $\epsilon > 0$ such that if $S \subset \cF$ is product-free and $\db(S) > 1/2 - \epsilon$, then there exists an odd-occurrence set $\cO_\Gamma$ such that $\db(S \setminus \cO_\Gamma) < \delta$?
\end{conjecture}

\Cref{thm:skpfstructure,thm:kpfstructure} give the structure of extremal (strongly) $k$-product-free sets in the free semigroup. The free group case remains. The simplest open case is the following

\begin{conjecture}
    Let $\cA$ be a finite set and $\fF$ be the free group with alphabet $\cA$. If $S \subset \fF$ is product-free and $\db(S) = 1/2$, then the following holds. It is possible to label each letter of $\cA \cup \cA^{-1}$ with a label in $\Ztwo$ such that the label of $\alpha^{-1}$ is the negation of the label of $\alpha$ for all $\alpha \in \cA$ and $S$ is a subset of
    \begin{equation*}
        T \coloneqq \set{w \in \fF \colon \text{the sum of the labels of letters in $w$ is $1 \bmod{2}$}}.
    \end{equation*}
\end{conjecture}

For strongly $k$-product-free we expect the above conjecture to hold with 2 replaced by $k$. For $k$-product-free we expect the behaviour to be the same as for the free semigroup.

We remark that our methods do give some structure. Similar arguments to \cref{sec:skpfstructure} show there is a labelling of all words in the subsemigroup $\fF^{\alpha \beta}$ (defined in \cref{sec:pfdensityfree}) such that the label of a concatenation is the sum of the individual labels and all words in $S \cap \fF^{\alpha \beta}$ have label 1. What is missing is an understanding of how the labellings interact when letters cancel during concatenation.

{
\fontsize{11pt}{12pt}
\selectfont
	
\hypersetup{linkcolor={red!70!black}}
\setlength{\parskip}{2pt plus 0.3ex minus 0.3ex}

}

\end{document}